\documentclass{ar-1col-S2O}
\usepackage[numbers,sort&compress]{natbib}
\usepackage{url} 
\usepackage{amsmath,amsthm,amssymb,amsfonts,bbm,bm,math}
\usepackage{epsfig,epsf,enumitem}
\usepackage{etoolbox}
\usepackage{graphicx}
\usepackage{kz_style}  
\usepackage{array}

\usepackage{amsbsy,comment} 

\DeclareMathOperator*{\conv}{conv}

\newcommand{\bmtx}{\begin{bmatrix}}
\newcommand{\emtx}{\end{bmatrix}}
\newcommand{\bsmtx}{\left[ \begin{smallmatrix}} 
\newcommand{\esmtx}{\end{smallmatrix} \right]} 
\newcommand{\bmatarray}[1]{\left[\begin{array}{#1}}
\newcommand{\ematarray}{\end{array}\right]}

\newtheorem{assumption}{Assumption}
\newtheorem{lemma}{Lemma}
\newtheorem{theorem}{Theorem}

\newtheorem{definition}{Definition}
\newtheorem{proposition}{Proposition}

\DeclareMathOperator*{\dom}{dom}

\makeatletter
\newcommand*{\rom}[1]{\expandafter\@slowromancap\romannumeral #1@}
\makeatother

\newcommand{\remove}[1]{}

\jname{Xxxx. Xxx. Xxx. Xxx.}
\jvol{AA}
\jyear{2022}
\doi{10.1146/((please add article doi))}

\begin{document}

\markboth{Hu, Zhang, Li, Mesbahi, Fazel, Ba\c{s}ar}{Policy Optimization for Learning Control Policies}

\title{Towards a Theoretical Foundation of Policy Optimization for Learning Control Policies
}

\author{Bin Hu$^1$, Kaiqing Zhang$^2$, Na Li$^3$, Mehran Mesbahi$^4$, Maryam Fazel$^5$, and Tamer Ba\c{s}ar$^6$
\affil{$^1$CSL \& ECE, University of Illinois
at Urbana-Champaign, IL, USA, 61801; email: binhu7@illinois.edu}
\affil{$^2$LIDS \& CSAIL, Massachusetts Institute of Technology,  Cambridge, MA, USA, 02139; ECE \& ISR, University of Maryland, College Park, MD, 20740; kaiqing@\{mit,umd\}.edu}
\affil{$^3$SEAS, Harvard University, Cambridge, MA, USA, 02138;nali@seas.harvard.edu}
\affil{$^4$AA, University of Washington, Seattle, WA, USA, 98195; mesbahi@uw.edu}
\affil{$^5$ECE, University of Washington, Seattle, WA, USA, 98195;mfazel@uw.edu}
\affil{$^6$CSL \& ECE, University of Illinois
at Urbana-Champaign, IL, USA, 61801; email: basar1@illinois.edu}}

\begin{abstract}
Gradient-based  methods have been widely used for system design and optimization in diverse application domains.   
Recently, there has been a renewed
interest in studying theoretical  properties of these methods in the context of control and reinforcement learning. 
This article surveys some of the recent developments  on policy optimization, a gradient-based iterative approach for feedback control synthesis, 
popularized by successes of reinforcement learning. 
We take  an interdisciplinary perspective in our exposition that
connects control theory, reinforcement learning, and large-scale optimization. We  review a number of recently-developed theoretical results on the optimization landscape, global convergence, and sample complexity of gradient-based methods for various continuous  control problems such as the linear quadratic regulator (LQR), $\mathcal{H}_\infty$ control, risk-sensitive control, linear quadratic Gaussian (LQG) control, and output feedback synthesis. In conjunction with these optimization results, we  also discuss how direct policy 
optimization handles stability and robustness concerns in learning-based control, two main desiderata in control engineering. 
We conclude the survey by pointing out several challenges and  opportunities at the intersection of 
learning and control.
\end{abstract}

\maketitle   


\section{Introduction}
Reinforcement learning (RL) has recently shown an impressive performance on a wide range of applications, 
from playing Atari~\cite{mnih2015human,alphastarblog} and mastering the game of Go ~\cite{silver2016mastering,silver2017mastering}, to complex robotic manipulations  \cite{rajeswaran2017learning,lillicrap2015continuous,schulman2015high}.
Key to RL success is the algorithmic framework  of policy optimization (PO), 
where the policy, mapping observations to actions, is parameterized and  directly optimized upon to improve system-level performance.
Mastering Go using PO (combined with techniques such as efficient tree-search) is particularly encouraging,\footnote{Go is considered a challenging game to master, partially as the number of
its legal board positions is significantly larger than the number of atoms in the observable universe.} as
the main idea behind the latter is rather straightforward -- when learning  
has been formalized as minimizing a certain cost as a function of the policy,
devise an iterative procedure on the policy to improve the objective. 
For example, in the policy gradient (PG) variant of PO,
when learning is represented as minimizing a (differentiable) cost $J(K)$ over the policy
$K$, the policy is improved upon via a gradient update of the form $K^{n+1}=K^n-\alpha \nabla J(K^n)$,
for some step size $\alpha$ (also  referred to as the learning rate) and data-driven evaluation of the cost gradient $\nabla J$ at each iterate $n$. In fact, PO provides an umbrella formalism for
not only policy gradient (PG) methods~\cite{sutton2000policy}, but also
actor-critic~\cite{konda2000actor},
trust-region~\cite{schulman2015trust}, proximal PO methods~\cite{schulman2017proximal}. 

More generally, PO provides a streamlined approach to learning-based system design.
For example, PO gives a general-purpose paradigm for 
addressing complex nonlinear dynamics with user-specified cost functions: 
for tasks involving nonlinear dynamics and complex design objectives, one can parameterize the policy as a neural network to be ``trained'' using  gradient-based methods 
to obtain a reasonable solution. The PO perspective can also be 
adopted for other insufficiently parameterized decision problems such
as end-to-end perception-based control~\cite{lee2019stochastic, yarats2019improving, yarats2021mastering}. In this setting, it might be desired
to synthesize a policy directly on images. As such,
one can envision parameterizing a mapping from  pixels (observation) to actions (decisions) as a neural network, and 
learn the corresponding policy using the PO formalism. Lastly, we mention the use of scalable gradient-based algorithms
to efficiently train nonlinear policies on many parameters, making PO suitable for high-dimensional tasks.
Computational flexibility and conceptual accessibility of PO have made it a main workhorse for modern~RL.

In yet another decision theoretic science, PO has a long history in control theory \cite{draper1951principles,whitaker1958design,kalman1960contributions,talkin1961adaptive,levine1970determination,makila1987computational};
in fact, it has been popular among control practitioners when the system model is poorly understood or parameterized.
Nevertheless, despite its generality and flexibility, PO formulation of control synthesis
is typically nonconvex and as such, challenging for obtaining strong performance certificates, rendering it unpopular amongst system theorists.
Since the 1980's,  convex reformulations or relaxations of control problems have become popular due to the
development of convex programming and related global convergence theory \cite{boyd2004convex}. 
It has been realized that many problems in optimal and robust control can be reformulated as convex programs, namely, semidefinite programs (SDP) \cite{befb94,Gahinet1994,Scherer2004}, or relaxed via sum-of-squares (SOS) \cite{sostools04,anderson2015advances}, expressed in terms of ``certificates," e.g.,
matrix inequalities that represent Lyapunov or dissipativity conditions. However, these formulations have limitations when there is deviation from the canonical synthesis problems, e.g., when
there are constraints on the structure of the desired control/policy. When convex reformulations are not available, PO assumes an important role as the main viable option. 
Examples of such scenarios include static output feedback problem \cite{rautert1997computational}, structured $\mathcal{H}_\infty$ synthesis~\cite{apkarian2006nonsmooth,apkarian2008mixed,noll2005spectral,saeki2006static,gumussoy2009multiobjective,arzelier2011h2},
and distributed control~\cite{maartensson2009gradient}, all of significant  importance in  applications. The PO framework is more flexible, as evidenced by the recent advances in deep RL. PO is also more scalable for high-dimensional problems as it does not
generally introduce extra variables in the optimization problems and enjoys a broader range of optimization methods as compared with the SDP or SOS formulations. 
However, as pointed out previously, nonconvexity of the PO formulation, even on relatively simple linear control problems, have made deriving theoretical guarantees for  direct policy optimization challenging, preventing the acceptance of PO as a  mainstream 
control design tool.

In this survey,
our aim is to revisit these issues from a modern optimization perspective, 
and provide a unified perspective on the recently-developed global convergence/complexity theory for PO in the context of control synthesis. Recent theoretical results on PO for particular classes of control synthesis problems, some of which are discussed in this survey, are not only exciting, but also lead to a new research thrust at the interface of control theory and machine learning. This survey includes control synthesis related to linear quadratic regulator theory~\cite{pmlr-v80-fazel18a,bu2019lqr,malik2019derivative,mohammadi2021convergence,furieri2020learning,li2021distributed,hambly2021policy,yang2019provably,jin2020analysis,mohammadi2020linear}, stabilization \cite{perdomo2021stabilizing,ozaslan2022computing,zhao2022sample}, linear robust/risk-sensitive control \cite{zhang2021policy,zhang2020stability,gravell2020learning,zhang2021derivative,zhao2021primal,zhang2021provably,Guo2022hinf,keivan2021model}, Markov jump linear quadratic control~\cite{jansch2020convergence,jansch2020policy,rathod2021global,jansch2022policy}, Lur'e system control  \cite{qu2021exploiting}, output feedback control \cite{feng2019exponential,fatkhullin2020optimizing,zheng2021analysis,duan2021optimization,duan2022optimization,mohammadi2021lack,hu2022connect}, and dynamic filtering \cite{umenberger2022globally}.
Surprisingly, some of these strong global convergence results for PO have been obtained 
in the {\em absence of convexity} 
in the design objective and/or the underlying feasible set.

These global convergence guarantees have a number of implications for learning and control. Firstly, these results facilitate examining other classes of synthesis problems in the same general framework. 
As it will be pointed out in this survey, there is an elegant geometry at play between certificates and controllers in the synthesis process, with immediate algorithmic implications.
Secondly, the theoretical developments in PO 
have created a renewed interest in the control community to examine synthesis of dynamic systems from a complementary perspective, that in our view, is more integrated with learning in general, and RL in particular. This will complement and strengthen the existing connections between RL and control \cite{bucsoniu2018reinforcement,recht2019tour,matni2019self}. Lastly, the geometric analysis of PO-inspired algorithms may shed light on issues in state-of-the-art policy-based RL, critical for deriving guarantees for any subsequent RL-based synthesis procedure for dynamic systems.  

This survey is organized to reflect our perspective -- and our excitement -- on how PO (and in particular PG) methods provide a streamlined approach for system  synthesis, and build a bridge between control and learning.
First, we provide the PO formulations for various control problems in \S\ref{sec:formulation}. Then
 we delve into the PO convergence theory on the classic linear quadratic regulator (LQR) problem in \S\ref{basic-lqr}. As it turns out, a key ingredient for analyzing LQR PO hinges on coerciveness of the cost function and its gradient dominance property (see \S3.2). These properties can then be utilized to devise gradient updates ensuring stabilizing feedback policies at each iteration, and convergence to the globally optimal policy. In \S\ref{beyond-lqr} we highlight some of the challenges
 in extending the LQR PO theory to other classes of problems, including the role of
coerciveness, gradient dominance, smoothness, and the landscape of the optimization problem.
The PO perspective is
then extended to more elaborate synthesis problems 
such as linear robust/risk-sensitive control, dynamic games,
and nonsmooth $\mathcal{H}_\infty$ state-feedback synthesis in \S\ref{sec:advance_cases}. 
Through these extensions, we highlight how variations on the general theme set by the LQR PO
theory can be adopted to address {lack of coerciveness} or nonsmoothness of the objective
in these problems while ensuring the convergence of the iterates to solutions of interest.
This is then followed by examining PO for control synthesis with partial observations,
and in particular, PO theory for linear quadratic Gaussian and output feedback in
\S\ref{advance:partial-observations}. 
Our discussion in \S\ref{advance:partial-observations}
underscores the importance of the underlying geometry of the policy landscape in
developing any PO-based algorithms.
Fundamental connections between PO theory and convex parameterization 
in 
control are discussed in \S\ref{sec:convex_para_grad_dom}.
In particular, it is shown how the geometry of policies and certificates are intertwined through appropriately constructed maps between nonconvex PO formulation of the synthesis problems and the (convex) semidefinite programming parameterizations. This provides a unified approach for analyzing PO in various control problems studied on a case-by-case basis so far. 
In \S\ref{challenges}, we present current challenges and our outlook for a comprehensive PO theory for synthesizing
dynamical systems that ensures stability, robustness, safety, and optimality;
and underscore the challenges in addressing synthesis problems in the face of partial observations, nonlinearities,
and for multiagent settings.
\S\ref{challenges} also examines further connections between PO theory and 
machine learning, and highlights the possibility of integrating model-based \cite{recht2019tour} and model-free methods to achieve the best of both worlds, illustrating
how the main theme of this survey fits within the big picture of learning-based control.

\section{Policy Optimization for Linear  Control: Formulation}\label{sec:formulation}

Control design can generally be formulated as a policy optimization problem of the form,
\begin{align}\label{eq:opt}
    \min_{K\in \mathcal{K}}~~J(K),
\end{align}
where the decision variable $K$ is determined by the controller parameterization (e.g., linear mapping, polynomials,
kernels, neural networks, etc.), the cost function $J(K)$ is some task-dependent control performance measure (e.g.,  tracking errors, closed-loop $\mathcal{H}_2$ or $\mathcal{H}_\infty$ norm, etc.), and the
feasible set $\mathcal{K}$ represents the class of controllers of interest, for example,
ensuring closed loop stability/robustness requirements. 
Such a PO formulation is general, and enables flexible policy parameterizations. For example, consider a modern deep RL setting where one wants to design a policy maximizing some task-dependent reward function for a complicated nonlinear system $x_{t+1}=f(x_t,u_t, w_t)$ with $(x_t,u_t,w_t)$ being the state, action, and disturbance triplet.
PO has served as the main workhorse for addressing such tasks. Specifically, one just needs to  parameterize the policy $K$ as a (deep) neural network and then apply iterative PO algorithms such as trust-region policy optimization (TRPO)~\cite{schulman2015trust} and proximal policy optimization (PPO) \cite{schulman2017proximal} to learn the optimal weights.

The focus of this survey article is the recently-developed  (global) convergence, complexity, and landscape theory of PO on classic control tasks  including LQR, risk-sensitive/robust control, and output feedback control. In this section, we formulate these linear control problems as PO via properly selecting $(K,J,\mathcal{K})$ in Equation~\ref{eq:opt}. 

\vspace{5pt} 
\noindent\textbf{Case I: Linear quadratic regulator (LQR).} There are several ways to formulate the LQR problem. For simplicity, we start by considering a discrete-time linear time-invariant (LTI) system $x_{t+1}=Ax_t+Bu_t$, where $x_t$ is the state and $u_t$ is the control action. 
The design objective is to choose the control actions $\{u_t\}$ to
minimize a quadratic cost function $J:=\mathbb{E}_{x_0\sim \mathcal{D}}\sum_{t=0}^\infty \left(x_t^\tp Q x_t+ u_t^\tp R u_t\right)$ with  $Q\succeq 0$ and $R\succ 0$ being pre-selected cost weighting matrices. In this setting, the only randomness stems from the initial condition $x_0$, which is sampled from a certain distribution $\mathcal{D}$ with a full rank covariance matrix. It
is well known that under some standard stabilizability and detectability assumptions the optimal cost is finite and can be achieved by a linear
state-feedback controller of the form $u_t=-Kx_t$. Therefore, we can formulate the LQR problem as a special case of the PO problem as in Equation~\ref{eq:opt}. Specifically, the decision variable $K$ is simply the feedback gain matrix. 
Under a fixed policy $K$, we have $u_t=-Kx_t$ for all $t$, and
the LQR cost can be rewritten as
$J(K)=\mathbb{E}_{x_0\sim \mathcal{D}}\left[\sum_{t=0}^\infty x_0^\tp ((A-BK)^\tp)^t (Q+K^\tp R K)(A-BK)^t x_0\right]$,
which is a function of $K$. 
This cost can also be computed as $J(K)=\tr(P_K \Sigma_0)$, where $\Sigma_0=\mathbb{E}x_0 x_0^\tp$ is the (full-rank) covariance matrix of $x_0$, and $P_K$ is the solution of the following Lyapunov equation:
\begin{align}\label{eq:lyapu1}
(A-BK)^\tp P_K (A-BK)+Q+K^\tp R K =P_K.
\end{align}
The above cost $J(K)$ is only well defined when the closed-loop system matrix $(A-BK)$ is Schur stable, i.e., when the spectral radius satisfies $\rho(A-BK)<1$. Therefore, one can define the feasible set $\mathcal{K}$  as,
\#\label{equ:def_stabilizing_K}
\mathcal{K}=\{K:\rho(A-BK)<1\}.
\#
Now we can see that the LQR problem is a special case of the PO problem in Equation~\ref{eq:opt}.
There are several other slightly different ways to formulate the LQR problem. In an alternative formulation,  we can add stochastic process noise and consider the LTI system,
\#\label{eq:LTI2}
x_{t+1}=Ax_t+B u_t+w_t,
\#
where
the disturbance $\{w_t\}$ is  a zero-mean i.i.d. process with a full rank covariance matrix $W$. The design objective is then to choose $\{u_t\}$ to minimize the time-average cost
\begin{align}\label{eq:averagecost}
J := \lim_{T \rightarrow \infty }\frac{1}{T}\mathbb{E} \left[\sum_{t=0}^{T-1} \left(x^\tp_t Q x_t + u^\tp_t R u_t\right)\right],
\end{align}
where $Q\succeq 0$ and $R\succ 0$
are pre-selected weighting matrices. 
Again, it suffices to parameterize the policy as $u_t=-Kx_t$. For a fixed policy $K$, the cost in Equation \ref{eq:averagecost} can be computed as $J(K)=\tr(P_K W)$, where $P_K$ is the solution for Equation \ref{eq:lyapu1}. Again, the cost is well defined only for $K$ satisfying $\rho(A-BK)<1$.
This setting leads to almost the same PO formulation as before.
Similarly, discounted LQR can be formulated as PO.

\vspace{5pt}\noindent\textbf{Case II: Linear risk-sensitive/robust control.} One can enforce risk-sensitivity and robustness via the formulation of linear exponential quadratic Gaussian (LEQG) \cite{jacobson1973optimal} or $\mathcal{H}_\infty$ control \cite{zhou1996robust}, respectively.  
For linear risk-sensitive control, we still consider the LTI system as in Equation \ref{eq:LTI2} with $w_t\sim \mathcal{N}(0,W)$ being an i.i.d. Gaussian noise,
and the design objective is to choose control actions $\{u_t\}$ to
minimize an exponentiated quadratic cost,
\begin{align}\label{equ:def_obj} 
J:=\limsup_{T\to\infty}~~\frac{1}{T}\frac{2}{\beta}\log\mathbb{E}\exp\bigg[\frac{\beta}{2} \sum_{t=0}^{T-1}\left(x_t^\tp Q x_t+u_t^\tp R u_t\right) \bigg],
\end{align}
where $\beta$ is the parameter quantifying the intensity of risk-sensitivity, and the expectation  is taken over the distribution for $x_0$ and $w_t$ for all $t\geq 0$. 
One typically chooses $\beta>0$ to make the control ``risk-averse.''  As $\beta\to 0$, the objective in Equation \ref{equ:def_obj} reduces to the LQR cost. The above LEQG problem is also a special case of the PO problem as in Equation \ref{eq:opt}. It is known that the optimal cost can be achieved by a linear state-feedback controller.  Again, one can just parameterize the controller as $u_t=-Kx_t$, where the gain matrix $K$ is the decision variable. Then the cost function can be specified as $J(K)=-\frac{1}{\beta}\log\det (I-\beta P_K W)$,
where $P_K$ is the unique stabilizing  solution to the  algebraic  Riccati equation,\footnote{The solution $P_K\succeq 0$ satisfies $\rho\big((A-BK)^\tp(I-\beta P_KW)^{-1}\big)<1$, and $W^{-1}-\beta P_K\succ 0$.} 
\begin{align*}
P_K=Q+K^\tp RK+(A-BK)^\tp\big[P_K- P_KW^{\frac{1}{2}}(-\beta^{-1}I+ W^{\frac{1}{2}}P_KW^{\frac{1}{2}})^{-1}W^{\frac{1}{2}}P_K\big](A-BK).
\end{align*}
Notice that in this case $J$ is well defined only when $K$ is in the following feasible set,
\begin{align}
\mathcal{K}=\left\{K: \rho(A-BK)<1, \,\mbox{and}\,\, \norm{(Q+K^\tp R K)^{\frac{1}{2}}(zI-A+BK)W^{\frac{1}{2}}}_\infty<\frac{1}{\sqrt{\beta}}\right\},
\end{align}
where $\norm{\cdot}_\infty$ denotes the $\mathcal{H}_\infty$ norm of a given discrete-time transfer function.
Hence the LEQG problem is a special case of the PO problem with $J$ and $\mathcal{K}$ as defined above. 
For LEQG, the $\mathcal{H}_\infty$  constraint $\norm{(Q+K^\tp R K)^{\frac{1}{2}}(zI-A+BK)W^{\frac{1}{2}}}_\infty<\frac{1}{\sqrt{\beta}}$ is implicitly required by the problem formulation.
Importantly, the LEQG problem can be viewed as a special case of the more general mixed $\mathcal{H}_2/\mathcal{H}_\infty$ design problem studied in robust control. In this survey article, we will cover two important robust control settings, namely the mixed $\mathcal{H}_2/\mathcal{H}_\infty$ design and the $\mathcal{H}_\infty$ state-feedback synthesis. 
For mixed $\mathcal{H}_2/\mathcal{H}_\infty$ design, consider the system below, where $w_t$ is the disturbance and $z_t$ is the controlled output:
\begin{align}
x_{t+1}=Ax_t+Bu_t+Dw_t,\quad z_t=Cx_t+E u_t.\label{equ:def_mixed_discret}
\end{align}
It is standard to assume $E^\tp [C \,\, E]=[0 \,\, R]$ for some $R\succ 0$.
The mixed design objective is to synthesize a linear state-feedback controller that minimizes an upper bound on the $\mathcal{H}_2$ cost and satisfies an additional $\mathcal{H}_\infty$-robustness requirement on the channel from $w_t$ to $z_t$. The $\mathcal{H}_\infty$ constraint is posed explicitly and is powerful in guaranteeing robust stability in the presence of any small gain type of uncertainty, including being time-varying, dynamic, or nonlinear. For the mixed design problem, the robustness constraint is directly enforced on $K$, and hence the feasible set $\mathcal{K}$ is modified as,
\begin{align}\label{equ:feasible_K}
\mathcal{K}=\left\{K: \rho(A-BK)<1, \,\mbox{and}\,\, \norm{(C-EK)(zI-A+BK)D}_\infty<\gamma\right\},
\end{align}
where $\gamma$ quantifies the robustness level. The smaller $\gamma$ is, the more robust the system is in the $\mathcal{H}_\infty$ sense (since it can tolerate the small gain uncertainty at the level $\frac{1}{\gamma}$ by the Small Gain Theorem). There exist several objective functions that upper-bound the $\mathcal{H}_2$ cost \cite{mustafa1989relations,mustafa1991lqg}, and a common one is  $J(K)=\tr(P_K DD^\tp)$, where $P_K$ is the solution  to the 
above
Riccati equation with $Q=C^\tp C$, $\gamma=1/\sqrt{\beta}$, $W=DD^\tp$. Notice that the mixed $\mathcal{H}_2/\mathcal{H}_\infty$ control aims at improving the average $\mathcal{H}_2$ performance while ``maintaining" a certain level of robustness by keeping the closed-loop $\mathcal{H}_\infty$ norm to be smaller than a pre-specified number. In contrast, the $\mathcal{H}_\infty$ state-feedback synthesis aims at ``improving" the system robustness and the worst-case performance via achieving the smallest closed-loop $\mathcal{H}_\infty$ norm. For simplicity, consider the LTI system $x_{t+1}=Ax_t+Bu_t+w_t$ initialized at $x_0=0$.
The design objective of $\mathcal{H}_\infty$ control is to choose $\{u_t\}$ to minimize the quadratic cost $J:=\sum_{t=0}^\infty (x_t^\tp Q x_t+u_t^\tp R u_t)$ in the presence of the worst-case $\ell_2$ disturbance satisfying $\sum_{t=0}^\infty \norm{w_t}^2\le 1$.
This problem can be reformulated as the PO problem
with the cost $J(K)$ being defined as the closed-loop $\mathcal{H}_\infty$ norm given below 
\begin{align}\label{eq:hinfcost}
J(K)=\sup_{\omega\in[0, 2\pi]}\lambda_{\max}^{1/2}\big((e^{-j\omega}I-A+BK)^{-\tp}(Q+K^{\tp}RK)(e^{j\omega}I-A+BK)^{-1}\big).
\end{align}
The reason is that the above cost actually satisfies 
\begin{align*}
J^2(K)=\max_{\sum_{t=0}^\infty \norm{w_t}^2\le 1} \sum_{t=0}^\infty x_t^\tp (Q+K^\tp R K) x_t=\max_{\sum_{t=0}^\infty \norm{w_t}^2\le 1} \sum_{t=0}^\infty (x_t^\tp Q x_t+u_t^\tp R u_t).
\end{align*}
The above cost is well defined only for $K$ satisfying $\rho(A-BK)<1$.
Therefore, minimizing the $\mathcal{H}_\infty$ cost function defined by Equation \ref{eq:hinfcost} over $\mathcal{K}$ given by Equation \ref{equ:def_stabilizing_K}  leads to a policy that minimizes the quadratic cost under the worst-case $\ell_2$ disturbance.

\vspace{5pt}\noindent\textbf{Case III: Linear quadratic Gaussian (LQG) and output feedback control.} Consider the following LTI  system which can only be partially observed:
\begin{subequations}\label{eq:LQG-K}
\begin{align}
x_{t+1}&=Ax_t+Bu_t+w_t,\\
y_t&=Cx_t+v_t.
\end{align}
\end{subequations}
Here, $w_t$ and $v_t$ are zero-mean white Gaussian noises
with covariance matrices $W\succeq 0$ and $V \succ 0$. At step $t$, one can only observe $y_t$, and the state $x_t$ is not directly measured.
The design objective is to choose actions $\{u_t\}$ to minimize the time-averaged cost defined in Equation \ref{eq:averagecost} given such partial observation information. Again, $Q\succeq 0$ and $R\succ 0$ 
are pre-selected weighting matrices. 
It is assumed that the pairs $(A,B)$ and $(A,W^{1/2})$ are controllable, and $(C,A)$ and $(Q^{1/2},A)$ are observable.
This problem can also be formulated as a special case of the PO formulation given by Equation \ref{eq:opt}.  Under our assumptions, it suffices to consider (full-order) dynamic controllers of the form,
\begin{equation}\label{eq:Dynamic_Controller}
    \begin{aligned}
         \xi_{t+1} = A_{K}\xi_t + B_{K}y_t,~~~~~~
        u_t = C_{K}\xi_t,
    \end{aligned}
\end{equation}
where $\xi_t$ is the internal state of the controller and has the same dimension as $x_t$.  For convenience, we encode the dynamic controller as,
\begin{align}
K:=\begin{bmatrix}
      0 & C_{K} \\
    B_{K} & A_{K}
    \end{bmatrix}.
\end{align}
The cost function $J(K)$ is well defined when the closed-loop system is stable, and hence, the feasible set should be specified as,
\begin{align}
\label{eq:LQG-constraint}
\mathcal{K}=\left\{K: \bmat{A & B C_K\\ B_K C & A_K}\,\,\mbox{is Schur stable} \right\}. 
\end{align}
For any $K\in\mathcal{K}$, the cost $J(K)$ can be represented as,
\begin{equation}\label{eq:LQG_cost_formulation_discrete}
J(K)
=
\operatorname{Tr}
\left(
\begin{bmatrix}
Q & 0 \\ 0 & C_{K}^\tp R C_{K}
\end{bmatrix} X_K\right)
=
\operatorname{Tr}
\left(
\begin{bmatrix}
W & 0 \\ 0 & B_{K} V B_{K}^\tp
\end{bmatrix} Y_K\right),
\end{equation}
where $X_{K}$ and $Y_{K}$ are the unique PSD 
solutions to the following Lyapunov equations,
\begin{subequations}
\label{eq:LQG_LyapunovX_discrete}
\begin{align}
X_{K} &= \begin{bmatrix} A &  BC_{K} \\ B_{K} C & A_{K} \end{bmatrix}X_{K}\begin{bmatrix} A &  BC_{K} \\ B_{K} C & A_{K} \end{bmatrix}^\tp +  \begin{bmatrix} W & 0 \\ 0 & B_{K}VB_{K}^\tp \end{bmatrix}, 
\\
Y_{K} &= \begin{bmatrix} A &  BC_{K} \\ B_{K} C & A_{K} \end{bmatrix}^\tp Y_{K}\begin{bmatrix} A &  BC_{K} \\ B_{K} C & A_{K} \end{bmatrix} +   \begin{bmatrix} Q & 0 \\ 0 & C_{K}^\tp R C_{K} \end{bmatrix}.
\end{align}
\end{subequations}
Thereby, the LQG design problem can be formulated as a special case of PO.  It is possible to use other control parameterizations and enforce more structures on $K$. 
This will lead to PO formulations for general output feedback control. Such formulations are particularly useful for decentralized control. 

\vspace{5pt}
For all three cases, the PO formulation is nonconvex in the policy space  ~\cite{peres1994alternate,pmlr-v80-fazel18a}. This is in contrast to convex reformulations of these problems.   
Next, we will review the recently-developed PO theory for Cases I, II, and III in Sections~\ref{basic-lqr}, \ref{sec:advance_cases}, and \ref{advance:partial-observations}, respectively.

\section{Case I: Global Convergence and Complexity of PO for  LQR } \label{basic-lqr}

LQR provides arguably the most fundamental optimal control formulation. A main challenge for the PO formulation of LQR is that the stability constraints are nonconvex in the policy space.  The global convergence and complexity of PO methods on LQR have not been established until very recently.  We review such results in this section.

\subsection{Background: Optimization and Complexity}\label{sec:math_background}

Consider the constrained optimization problem $\min_{K\in\mathcal{K}} J(K)$ with $\mathcal{K}$ being nonconvex. 
If the feasible set $\mathcal{K}$ is open and the optimal value of $J$ is achieved by some interior point $K^*$ in $\mathcal{K}$, then we have  $\nabla J(K^*)=0$ (in this case, the KKT condition reduces to the first-order optimality condition for unconstrained problems), and it is possible to solve for $K^*$ via applying an iterative gradient-based algorithm with the update rule 
$K^{n+1}=K^n-\alpha F^n$, where $K^n$ denotes the controller parameter at iteration $n$, and $F^n$ is some descent direction of the cost $J$. The most common example of $F^n$ is the gradient direction $\nabla J(K^n)$ at~$K^n$. Some other examples include the natural gradient direction, Gauss-Newton (or other quasi-Newton) directions (see \S\ref{sec:LQR_results} for more details). 
 It is important to know whether and how fast $\{K^n\}$ converges to $K^*$. We will introduce one important optimization result for coercive and/or gradient-dominant function $J(K)$. 
 
 \begin{definition}[Coercive and gradient dominant properties]\label{def_coer} We call a function $J(K)$ 
 \begin{itemize}
     \item[i)] \textit{coercive} on $\mathcal{K}$ if for any sequence $\{K^l\}_{l=1}^\infty\subset \mathcal{K}$ we have
\begin{equation*}
    J(K^l) \rightarrow +\infty 
\end{equation*}
if either $\|K^l\|_2 \rightarrow +\infty$, or  $K^l$ converges to an element on the boundary $\partial \mathcal{K}$.
\item[ii)] $\mu$-\textit{gradient dominant} of degree $p$, if it is continuously differentiable and satisfies
\begin{align}\label{eq:grad_dominance}
J(K)-J(K^*)\le \frac{1}{2\mu}\norm{\nabla J(K)}_F^p, \quad \forall K\in \mathcal{K}, 
 \end{align}
 where $\mu$ is some positive constant, and $K^*$ is an optimal solution of $J(K)$ over  $\mathcal{K}$.\footnote{This property (also referred to as the Polyak- \L ojasiewicz (PL) Condition) appears commonly in the optimization literature \cite{nesterov2006cubic,karimi2016linear,LiPong2018KLexponent}, but is often used only locally. Here we are interested in special problems where this property holds globally.}
 \end{itemize}
 \end{definition}

 If $J(K)$ is coercive, then it serves as a barrier function over the feasibility set $\mathcal{K}$, and hence projection is not needed for maintaining feasibility. In addition, gradient dominance is useful for establishing global convergence.
 The following optimization result is fundamental and useful.

\begin{theorem}\label{theorem:main_technical_thm}
Suppose $J(K)$ is coercive.  Assume further that
$J$ is twice continuously differentiable over $\mathcal{K}$. Then the following statements hold:  
\begin{enumerate}
\item The sublevel set $\mathcal{K}_\gamma:=\{K\in\mathcal{K}:J(K)\le \gamma\}$ is compact.
\item The function $J(K)$ is $L$-smooth on $\mathcal{K}_\gamma$, and the constant $L$ depends on $\gamma$ and the problem parameters. Specifically, for any $(K, K')$ satisfying $tK+(1-t)K'\in \mathcal{K}_\gamma$ $\forall t\in [0,1]$, the following inequality holds
\begin{align}
\label{eq:smooth1}
    J(K')\le J(K)+ \langle\nabla J(K),(K'-K)\rangle+\frac{L}{2}\|K'-K\|_F^2.
\end{align}
\item Consider the gradient descent method 
\#\label{equ:PG_update}
K^{n+1}=K^n-\alpha \nabla J(K^n).
\# 
Suppose $K^0\in\mathcal{K}$.
Let $\gamma_0=J(K^0)$.
Suppose $L$ is the smoothness constant of $J(K)$ on $\mathcal{K}_{\gamma_0}$. Then, for any $0<\alpha<\frac{2}{L}$, we have $K^n\in \mathcal{K}$ for all $n$. In addition, we have $\nabla J(K^n)\rightarrow 0$ and the following convergence rate bound holds with $C=\alpha-\frac{L\alpha^2}{2}>0$.
\begin{align}
    \min_{0\le l \le k} \norm{\nabla J (K^l)}_F^2\le \frac{\gamma_0}{C(k+1)}.
\end{align}
\item If the function $J$ also satisfies the \emph{gradient dominance} property with degree $2$, then we have the linear convergence 
\begin{align}\label{eq:con1}
J(K^n)-J(K^*)\le (1-2\mu\alpha+\mu L\alpha^2)^n (J(K^0)-J(K^*)).
\end{align}
\end{enumerate}
\end{theorem}
\begin{proof}
This result is important, so a proof is included for illustrative purposes.
Statement~1 can be proved using  the continuity and coerciveness of $J(K)$, and is actually a direct consequence of \cite[Proposition 11.12]{bauschke2011convex}.

Since $J$ is twice continuously differentiable,  we know that the function $\norm{\nabla^2 J(K)}$ (with $\norm{\cdot}$ being the operator norm) is continuous. By Weierstrass theorem, we know $\norm{\nabla^2 J(K)}$ has to be bounded on the compact set $\mathcal{K}_\gamma$. We denote this uniform upper bound as $L$, and hence $J$ is $L$-smooth on $\mathcal{K}_\gamma$.  By mean value theorem, Equation \ref{eq:smooth1} holds as desired. This proves Statement 2.

The proof of Statement 3 is based on smoothness, and we will use standard arguments.
Suppose we have chosen $\alpha=\frac{2}{L+\omega}$ for some positive constant $\omega>0$. 
First, we need to show that given $K\in \mathcal{K}_{\gamma_0}$, the line segment connecting $K$ and $K'=K-\alpha \nabla J(K)$ is also in $\mathcal{K}_{\gamma_0}$. By continuity of $\norm{\nabla^2 J(K)}$ and $J(K)$, there exists a small constant $c>0$ such that $\norm{\nabla^2 J(K)}\le L+\omega$ for all $K\in \mathcal{K}_{\gamma_0+c}$. Denote the closure of the complement of $\mathcal{K}_{\gamma_0+c}$ as $S_1$. Obviously,  $\mathcal{K}_{\gamma_0}\cap S_1$ is empty. Since $\mathcal{K}_{\gamma_0}$ is compact, 
we know the distance between $\mathcal{K}_{\gamma_0}$ and $S_1$ is strictly positive. We denote this distance as $\delta$. Let us choose $\tau=\min\{0.9\delta/\norm{\nabla J(K)}_F \,,2/(L+\omega)\}$. Clearly, the line segment between $K$ and $(K-\tau\nabla J(K))$ is in $\mathcal{K}_{\gamma_0+c}$. Notice that  $\|\nabla^2 J(K)\|\le L+\omega$ for all $K\in \mathcal{K}_{\gamma_0+c}$, and hence we have
\begin{align*}
J(K-\tau\nabla J(K)) \le  J(K)+ \langle\nabla J(K), K-\tau\nabla J(K)-K\rangle
+\frac{L+\omega}{2}\|K-\tau\nabla J(K)-K\|_F^2,
\end{align*}
which leads to $J(K-\tau\nabla J(K))\le  J(K)+ \left(-\tau+\frac{(L+\omega)\tau^2}{2}\right)\|\nabla J(K)\|_F^2$.
As long as  $\tau\le 2/(L+\omega)$, we have $-\tau+\frac{(L+\omega)\tau^2}{2}\le 0$ and $J(K-\tau\nabla J(K))\le J(K)\le \gamma_0$. Hence we have $K-\tau\nabla J(K)\in \mathcal{K}_{\gamma_0}$. Actually, it is straightforward to see that the line segment between $K$ and $(K-\tau \nabla J(K))$ is in $\mathcal{K}_{\gamma_0}$ by varying $\tau$.
The rest of the proof follows from induction.
We can apply the same argument to show that the line segment between $(K-\tau\nabla J(K))$ and $(K-2\tau\nabla J(K))$ is also in $\mathcal{K}_{\gamma_0}$. This means that the line segment between $K$ and $(K-2\tau\nabla J(K))$ is in $\mathcal{K}_{\gamma_0}$. Since $\tau>0$, we only need to apply the above argument for finite times and then will be able to show that the line segment between $K$ and $(K-\alpha \nabla J(K))$ is in $\mathcal{K}_{\gamma_0}$ for $\alpha= \frac{2}{L+\omega}$. 
Now we can apply Equation \ref{eq:smooth1} to show the convergence result.
Since $\|\nabla^2 J(K)\|\le L$ for all $K\in \mathcal{K}_{\gamma_0}$, we can use the mean value theorem to show
\begin{align*}
    J(K')\le  J(K)+ \langle \nabla J(K), K'-K\rangle+\frac{L}{2}\|K'-K\|_F^2= J(K)+\left(-\alpha+\frac{L\alpha^2}{2}\right)\|\nabla J(K)\|_F^2,
\end{align*}
which can be summed over the window from $0$ to $k$ to get the desired convergence result.

Finally, we can combine the gradient dominance  inequality with the above smoothness inequality to show $J(K')-J(K)\le -(2\mu\alpha-\mu L\alpha^2)(J(K)-J(K^*))$. This immediately leads to Equation \ref{eq:con1}. Our proof is complete.
\end{proof}
Next, we will show that the convergence/complexity of the gradient descent method on LQR follows as a consequence of the above result.

\subsection{PO Theory for LQR}\label{sec:LQR_results} 
There are multiple ways to show
the global convergence/complexity of the gradient descent method on the LQR problem \cite{pmlr-v80-fazel18a,bu2019lqr,fatkhullin2020optimizing,mohammadi2019global}. In this section, we
review one proof which is based on Theorem \ref{theorem:main_technical_thm}. Interestingly, the LQR cost is coercive, real analytic, and gradient dominant so that Theorem \ref{theorem:main_technical_thm} can be directly applied, though the problem is nonconvex in the  parameter $K$ \cite{peres1994alternate,pmlr-v80-fazel18a}. 
Recall that the LQR cost can be computed as $J(K)=\tr(P_K \Sigma_0)$, where $\Sigma_0=\mathbb{E}x_0 x_0^\tp$ and $P_K$ satisfies
$(A-BK)^\tp P_K (A-BK)+Q+K^\tp R K =P_K$. 
For simplicity, we assume here that $Q\succ 0$, following \cite{pmlr-v80-fazel18a}\footnote{The results can be generalized to the cases where $Q\succeq 0$ \cite{bu2019lqr} and even  $Q$ being indefinite \cite{bu2020global}.}. The following result holds.

\begin{lemma}\label{lemma:coercive}
The LQR cost satisfies the following properties. The cost function $J$ is:  
\begin{enumerate}
   \item Real analytical and hence twice continuously differentiable. 
    \item Coercive over the feasible set $\mathcal{K}$. 
\item $\mu$-gradient dominant with 
$\mu=\frac{2(\sigma_{\min}(\mathbb{E}x_0 x_0^\tp))^2\sigma_{\min}(R)}{\norm{\Sigma_{K^*}}}$, where $\sigma_{\min}$ and $\norm{\cdot}$ denote the smallest and largest singular values, respectively.
\end{enumerate}
\end{lemma}
\begin{proof}
To prove Statement 1, notice that the analytical solution of the Lyapunov equation can be calculated as
$\vect(P_K)=\left(I-(A-BK)^\tp\otimes (A-BK)^\tp\right)^{-1}\vect(Q+K^\tp R K)$.
Hence $P_K$ is a rational function of the elements of $K$. Then
we know that $J$ is a rational function of the elements of $K$. Therefore, $J$ is real analytical and twicely continuously differentiable. 

To prove Statement 2, one can apply the contradiction argument in \cite{bu2019lqr}.  We refer readers to  \cite{bu2019lqr} for details of this argument.

Finally, Statement 3 can be proved using 
the cost difference lemma \cite[Lemma 10]{pmlr-v80-fazel18a}. See \cite[Lemma 11]{pmlr-v80-fazel18a} for one such argument\footnote{The constant coefficient used in \cite[Lemma 11]{pmlr-v80-fazel18a} can be slightly tightened to match the exact value of $\mu$ given in Lemma \ref{lemma:coercive}.}.
\end{proof}

It is worth mentioning that 
we can explicitly bound the smoothness constant $L$ over any sublevel set of $J(K)$ in terms of problem paramaters, which helps 
in establishing refined convergence rates 
for the gradient descent method. 
We are now ready to state  the global convergence result for the gradient descent method on LQR. 

\begin{theorem}\label{thm:LQR-PG}
Consider the LQR PO problem with $(Q,R)$ being positive definite, and apply the gradient method with the update rule $K^{n+1}=K^n-\alpha \nabla J(K^n)$. Suppose $K^0\in\cK$ is stabilizing, then with stepsize satisfying  $0<\alpha\leq 1/L_{K^0}$, where $L_{K^0}$ denotes the smoothness constant of $J(K)$ over the sublevel set of level $J(K^0)$, we have that: i) for all $n\geq 1$, $K^n$ stabilizes  the system, i.e., $\rho(A-BK^n)<1$; ii) the sequence $\{K^n\}$ converges to the global optimum of LQR at a linear rate as
$$
J(K^n)-J(K^*)\leq \Big(1-\mu\alpha\Big)^n\big(J(K^0)-J(K^*)\big),
$$
where $\mu$ is the gradient dominance coefficient given in Lemma \ref{lemma:coercive}. 
\end{theorem}

The proof of Theorem \ref{thm:LQR-PG} follows 
Theorem \ref{theorem:main_technical_thm} and Lemma \ref{lemma:coercive}. The above result requires an initial stabilizing controller $K^0\in\cK$. This is not an issue, since it is also known that one can obtain such stabilizing policies using PO methods \cite{lamperski2020computing,perdomo2021stabilizing,ozaslan2022computing,zhao2022sample}.

\vspace{5pt} 
\noindent\textbf{Zeroth-order optimization.} 
In many cases, the exact gradient $\nabla J(K)$ is not available, especially when the dynamical system model is unknown.  In the optimization and learning community, one method which has been actively studied is to estimate the gradient through the cost value, $J(K)$. For example, we can estimate the gradient using the following single-point zeroth-order gradient estimator:
\begin{equation}\label{eq:single_point}
\mathsf{G}_J(K;r,z)=\frac{d}{r} J(K+rz)\,z,
\qquad z\sim\mathcal{Z}.
\end{equation}
Here $r>0$ is a positive parameter called the \emph{smoothing radius}. We slightly abuse the notation  by letting $z$ denote the \emph{random perturbation}, which is a $d$-dimensional random vector following the probability distribution $\mathcal{Z}$. Usually, $\mathcal{Z}$ is chosen to be one of the following: i) the Gaussian distribution $\mathcal{N}(0,d^{-1}I)$.
ii) The uniform distribution on the unit sphere $\mathbb{S}_{d-1}\coloneqq \{z\in\mathbb{R}^d: \|z\|=1\}$, which we denote by $\operatorname{Unif}(\mathbb{S}_{d-1})$.  For this single-point estimator, it is shown that $
\mathbb{E}_{z\sim\mathcal{Z}}\!\left[\mathsf{G}_J(K;r,z)\right]
=\nabla J_r(K),
$ where $J_r$ is a smoothed version of $J$ 
and the radius $r$ controls the approximation accuracy. 
Besides the single point estimator given in Equation~\ref{eq:single_point},  multi-point gradient estimators can be used to improve the convergence rate \cite{nesterov2017random, duchi2014twopoint}. 
The optimization and learning literature, e.g., \cite{shamir2013complexity,ghadimi2013stochastic, nesterov2017random,balasubramanian2022zeroth,tang2020zeroth} have studied the properties/complexity of the zeroth-order methods under different settings, e.g., convex or nonconvex $J(\cdot)$, stochastic or deterministic optimization, etc. When using the zeroth-order methods for LQR, we need to pay extra attention to the following issues:
\begin{enumerate}[leftmargin=12pt]
    \item Initial feasible $K^0$:
    Same as in the exact gradient case, the initial $K^0$ should be feasible, i.e., stabilizing the system. If a system is unknown, indeed this is challenging.  
    Recent works  \cite{perdomo2021stabilizing,ozaslan2022computing,zhao2022sample} provide the convergence/complexity theory for using PO-based discount annealing methods
    to obtain initial stabilizing policies.  
    \item Impact of bias and variance of zeroth-order estimation on the feasibility and convergence of $K^n$: Though Theorem~\ref{thm:LQR-PG} ensures feasibility for gradient descent iterates, zeroth-order estimation introduces both bias and variance in the gradient evaluation.  
    To address this, one can tune the parameter $r$, and use the average of several single-point estimators. 
    \item Feasibility of the perturbed controller $K^n+rz$: 
    We need to ensure that the iterate $K^n+rz$ is in $\mathcal{K}$. This limits the choices of random exploration $rz$, 
    e.g., $K^n$ should be at least strictly feasible to allow perturbations.  
    \item Evaluating $J(K)$:  
    When the cost is defined on the infinite time horizon (e.g., in LQR), it is challenging to evaluate in practice; often one can obtain only a
    finite-time truncated estimate for $J(K)$. Handling the truncation requires care, as it also introduces bias and variance in the estimator. 
    \item Dependence of the sample complexity on system  parameters:  \cite{pmlr-v80-fazel18a,malik2019derivative} show that the sample complexity of LQR zeroth-order methods depends on the system parameters $A, B, Q, R$ and the initial controller $K^0$. If the system is ``ill-conditioned'', 
    the number of samples can potentially be quite large, as discussed in the recent work~\cite{ziemann2022policy}. See more general discussions on the interplay between statistical learning theory and control in the concurrent survey \cite{tsiamis2022statistical}. 
\end{enumerate}

Due to space limitations, we refer readers to  \cite{pmlr-v80-fazel18a,malik2019derivative,li2021distributed}  for details on how the above issues were handled when implementing zeroth-order methods for LQR.
There are other data-driven policy gradient estimation methods such as the policy gradient theorem~\cite{sutton2000policy,schulman2015high} or iterative feedback tuning~\cite{hjalmarsson1998iterative,hjalmarsson2002iterative}. Sample complexity for these methods are less understood. 
We end this section with a few remarks on other aspects of the LQR PO~theory.

\vspace{5pt}\noindent\textbf{LQR with stochastic noise.} The above convergence result extends to more general forms of LQR, e.g., with process noise as in Equation \ref{eq:LTI2}. The major change in the derivations is to replace the matrix $\mathbb{E}(x_0 x_0^\tp)$ with the covariance of the process noise. 

\vspace{5pt}\noindent\textbf{Natural policy gradient (NPG).}  NPG is a standard RL algorithm which enforces a KL divergence constraint on the updated and the old policies \cite{kakade2002natural}. In the LQR setting, the deterministic counterpart of NPG is given as follows
\begin{align}\label{eq:NPG_LQR}
K^{n+1}=K^n-\alpha \nabla J(K^n)\Sigma_{K^n}^{-1},
\end{align} 
where $\Sigma_{K^n}$ is the state correlation matrix for $K^n$ (see \cite{pmlr-v80-fazel18a} for details).
For a discounted LQR problem with a stochastic Gaussian policy, the NPG update (which inverts the Fisher information matrix) exactly reduces to the above iterative scheme. 
The NPG method in Equation \ref{eq:NPG_LQR} also converges at a linear rate. If we denote $F^n=\nabla J(K^n)\Sigma_{K^n}^{-1}$, then we have
$J(K^n)-J(K^*)\le \frac{\norm{\Sigma_{K^*}}}{\sigma_{\min}(R)}\tr((F^n)^\tp F^n)$,
which can be combined with the cost difference lemma (see Lemma 10 in \cite{pmlr-v80-fazel18a}) to show the linear convergence of NPG.

\vspace{5pt}\noindent\textbf{Policy iteration \cite{bradtke1994adaptive} and Kleinman's algorithm~\cite{kleinman1968iterative}.} An important variant of the gradient descent method is the Gauss-Newton method with the following  iterates 
\begin{align}\label{equ:gauss_newton_update}
K^{n+1}=K^n -\alpha (B^\tp P_{K^n} B+R)^{-1} \nabla J(K^n) \Sigma_{K^n}^{-1},
\end{align}
which is equivalent to
$K^{n+1}=K^n-2\alpha\left(K^n-(B^\tp P_{K^n} B+R)^{-1} B^\tp P_{K^n} A\right)$.
If $\alpha=\frac{1}{2}$, this algorithm reduces to the policy iteration (PI) algorithm in the RL literature, or equivalently Kleinman's algorithm in the controls literature \cite{kleinman1968iterative,hewer1971iterative}. We can show that the Gauss-Newton method has a global linear convergence rate for any $\alpha\le \frac{1}{2}$. In addition, when $\alpha=\frac{1}{2}$, the above method has a superlinear local rate. This explains why PI is typically fast on the LQR problem. PI can be implemented in a model-free manner via least-squares techniques \cite{lagoudakis2003least}. There are also sample complexity results for approximate PI on LQR \cite{krauth2019finite}. 

\vspace{5pt}\noindent\textbf{Further extensions.}
  The above LQR PO theory can also be extended to cover: i) time-varying/nonlinear systems such as Markovian jump linear systems~\cite{jansch2020convergence,jansch2020policy,rathod2021global,jansch2022policy} and Lur'e systems \cite{qu2021exploiting}, ii) more complicated RL algorithms, e.g. 
actor-critic \cite{yang2019provably,jin2020analysis}, and iii) different settings including the continuous-time setting \cite{mohammadi2019global,mohammadi2020linear,mohammadi2021convergence,bu2020policy},  the multiplicative noise setting \cite{gravell2020learning}, and the finite-horizon setting \cite{hambly2021policy}.

\subsection{Technical Challenges for Settings  Beyond LQR} \label{beyond-lqr}

The global convergence of PO for LQR heavily relies on several important properties of the cost function and feasible set. 
We summarize the importance of these properties as follows.

\vspace{5pt}\noindent\textbf{Coerciveness of cost.} A key factor to the results for LQR is the coerciveness of the objective, as stated in Lemma  \ref{lemma:coercive}. Coerciveness ensures that as long as the cost function value decreases, the controller $K^n$ remains stabilizing, i.e., feasible.  Furthermore, coerciveness  ensures that the sublevel sets of the cost function are {\it compact}, which, together with the real analytical property of the cost, implies that the gradient of the cost is {globally Lipschitz} over any finite level set, i.e., the cost is globally smooth over its sublevel sets. This smoothness property serves as one of the pillars in nonconvex optimization  analysis, in determining the stepsize that sufficiently decreases the cost, see e.g., \cite{bertsekas1997nonlinear}. The coercive property makes  the cost function a valid barrier function that explicitly {\it regularizes} the iterates to be feasible along the iterations. However, the cost is not necessarily coercive for other control problems, and only decreasing the cost value thus  may 
no longer ensure the feasibility of the iterates.

\vspace{5pt}\noindent\textbf{Gradient dominance (PL property) of cost.} Convergence to the global minimum of policy gradient methods for LQR, especially with {\it linear} convergence rate, relies heavily on the benign landscape property of gradient dominance (cf. Lemma \ref{lemma:coercive}). 
Together with the smoothness of the objective, the gradient dominance property (of degree $2$, see 
Definition~\ref{def_coer})  naturally leads to global convergence rate that can be linear \cite{karimi2016linear}. This benign property is a blessing for  certain control problems (see  \S\ref{sec:convex_para_grad_dom}), and does not necessarily hold in general. 

\vspace{5pt}\noindent\textbf{Smoothness of cost.} Sometimes the cost may not be differentiable over the entire feasible~set. For example, for optimal $\mathcal{H}_\infty$ control, the cost function can be non-differentiable at stationary points \cite{apkarian2006nonsmooth,burke2020gradient}. The lack of smoothness causes difficulty for such PO problems.   

\vspace{5pt}\noindent\textbf{Connectivity of feasible set.} Another key to the success of PO for LQR, as a {\it local search} approach, is that the feasible set, though nonconvex in general, is connected. This is important since local search algorithms typically cannot jump between connected components, and a single connected component has to include the global optimum. Unfortunately,  such connectivity is lost when extending PO to partially observable control systems, imposing additional challenges in establishing global convergence.

Next, we study several  control problems where some (if not all) of these desired properties are lacking, and more careful and advanced analyses are needed in order to obtain global convergence guarantees for PO methods.

\section{Case II: PO for Risk-Sensitive \& Robust  Control}\label{sec:advance_cases}

In this section, we review the global convergence results of PO methods on the mixed $\mathcal{H}_2/\mathcal{H}_\infty$ control problem and the state-feedback $\mathcal{H}_\infty$  optimal control problem. For the mixed design problem, the main issue is the lack of coerciveness, i.e., the cost function close to the boundary of the feasible set may not approach infinity. For the  $\mathcal{H}_\infty$ synthesis problem, the main difficulty is the lack of smoothness, i.e., the cost function may be non-differentiable over some important points in the feasible set. 
We will discuss how to modify PO algorithms to mitigate these issues and achieve global convergence provably. 
It is worth emphasizing that the idea of applying RL methods to solve $\mathcal{H}_\infty$ control is not new \cite{wu2012neural,luo2014off,kiumarsi2017h}. 
 This section mainly focuses on the recently-developed global convergence theory for PO methods on such robust control tasks \cite{ zhang2021policy,zhang2020stability,Guo2022hinf}.

\subsection{PO for Mixed $\mathcal{H}_2/\mathcal{H}_\infty$ Design: Implicit Regularization}\label{sec:robust_control_results} 

Recall the formulation of linear risk-sensitive and mixed design problems in {\bf Case II} in \S\ref{sec:formulation}. For simplicity, we use one common objective of the problem, and restate it as follows 
\#\label{equ:def_mixed_formulation}
&\qquad\qquad\qquad\min_{K}~\qquad~~J(K):=\tr(P_K DD^\tp)\\
\quad & s.t.~~~~\qquad\qquad\qquad\qquad K\in \cK \text{~~in~~Equation \ref{equ:feasible_K}} \text{~~and~~}\notag\\
&(A-BK)^\tp(P_K+P_KD(\gamma^2I-D^\tp P_K D)^{-1}D^\tp P_K)  (A-BK)+C^\tp C+K^\tp RK-P_K=0.\notag
\#
The above cost function $J(K)$ is known to be differentiable over the feasible set. One may wonder whether the analysis for the LQR case can be tailored to establish the global convergence of the gradient descent method on the above mixed $\mathcal{H}_2/\mathcal{H}_\infty$ design problem.
However, the following lemma reveals the less desired landscape properties of the cost~function.

\begin{lemma}[Nonconvexity and No Coerciveness]\label{thm:ncnc_property}
 	The feasible set for the mixed $\cH_{2}/\cH_{\infty}$ design  problem in Equation \ref{equ:def_mixed_formulation}   is nonconvex. Moreover, the cost function given in Equation~\ref{equ:def_mixed_formulation}  is not coercive. Particularly, as $K\to \partial \cK$, where $\partial \cK$ is the boundary of the constraint set $\cK$, the cost $J(K)$ does not necessarily  approach  infinity.
 \end{lemma} 
 
 \begin{figure}[!t]
	\centering
	\begin{tabular}{cc}
		\hskip-115pt
		\includegraphics[width=0.28\textwidth]{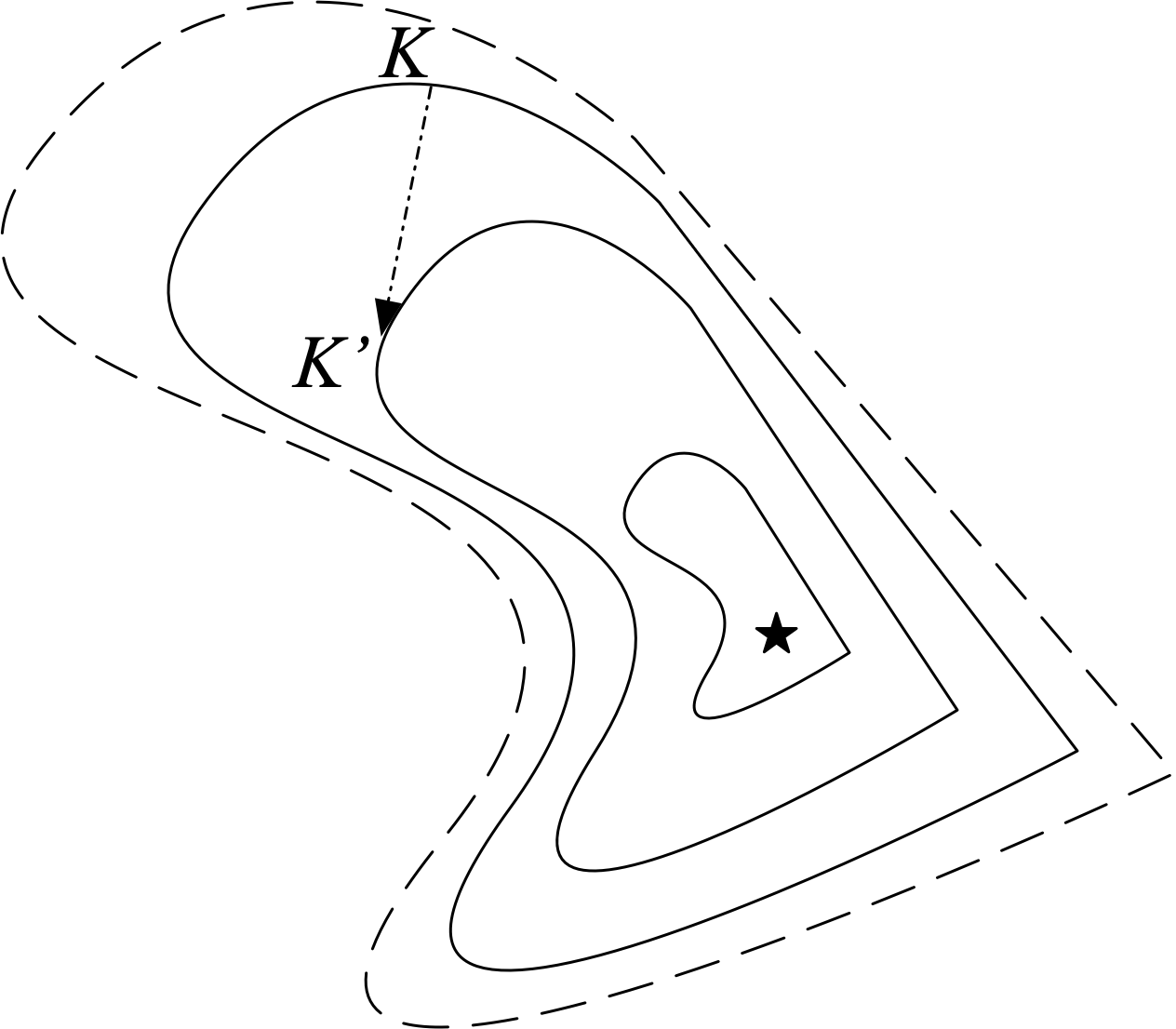}
		&
		\hskip-190pt
		\includegraphics[width=0.3\textwidth]{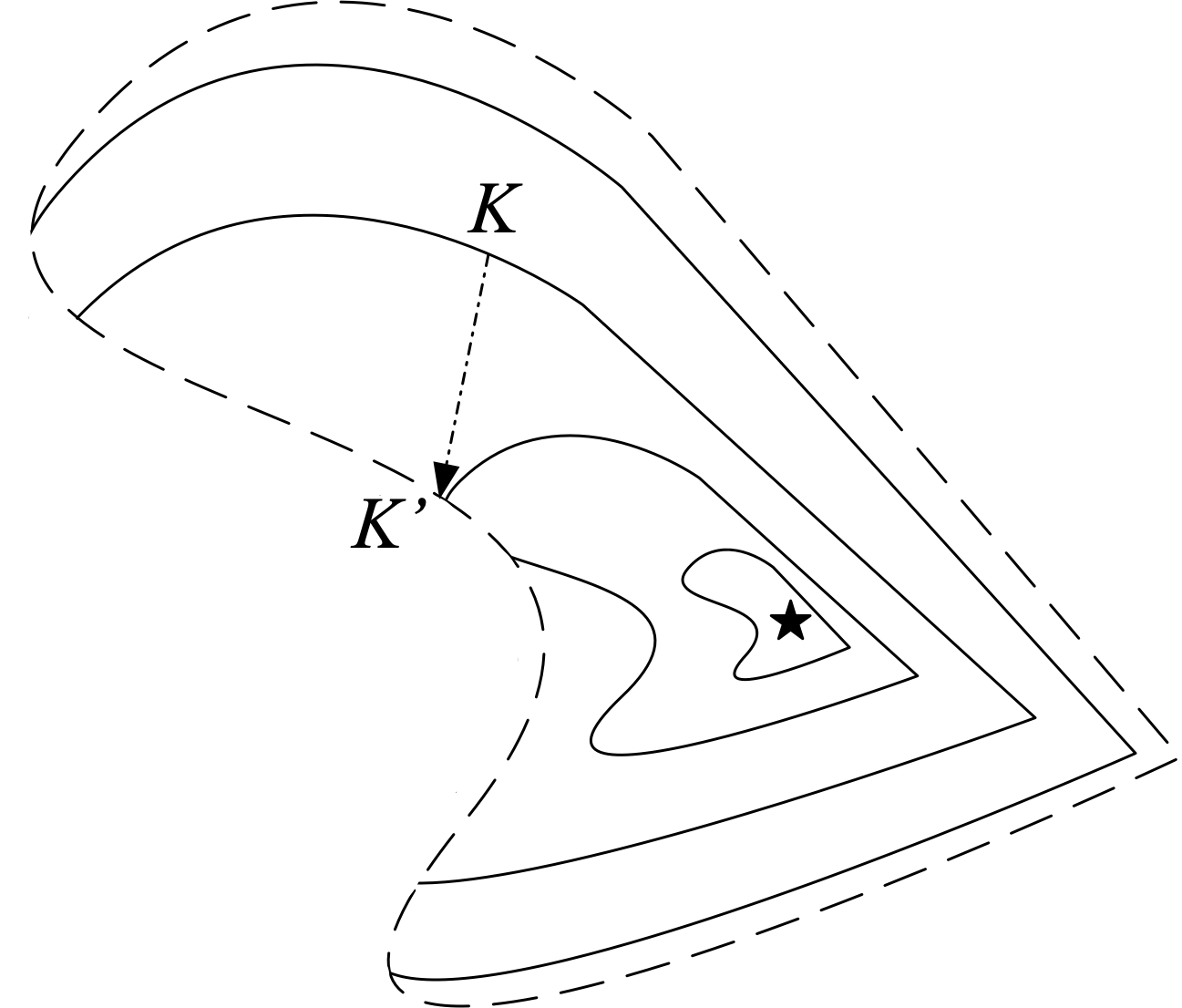}
		\\
		\hskip-115pt(a) Landscape of LQR & \hskip -190pt(b) Landscape of Mixed  $\cH_2/\cH_\infty$  Control
	\end{tabular}
	\vspace{7pt}
	\caption{Comparison of the landscapes of LQR and mixed $\cH_2/\cH_\infty$ control design. The dashed lines represent the boundaries of the constraint sets $\mathcal{K}$. For (a) LQR, $\mathcal{K}$ is the set of all linear stabilizing state-feedback controllers; for (b) mixed $\cH_2/\cH_\infty$ control,  $\mathcal{K}$ is set of all linear stabilizing state-feedback controllers satisfying an extra $\cH_\infty$ constraint. The solid lines are the contour lines of the cost $J(K)$. $K$ and $K'$ denote the control gains of two consecutive iterates; {\scriptsize $\bigstar$} denotes the global optimizer.}
	\label{fig:illust_hardness}
\end{figure}

 The difference between the landscapes of LQR and mixed $\cH_2/\cH_\infty$ control is illustrated in Figure \ref{fig:illust_hardness}.  The cost function for mixed $\cH_2/\cH_\infty$ control is not necessarily coercive and hence cannot serve as a barrier function over the feasible set by itself; see further discussion in
\cite{zhang2021policy}. The lack of coerciveness for the mixed design problem calls for a more careful analysis on maintaining the feasibility of the iterates during optimization. In \cite{zhang2021policy}, the concept of {\it implicit regularization} has been adopted to address this issue. Specifically, a PO algorithm is referred to as being  \emph{implicitly regularized} if  the iterates $\{K_n\}$ generated by the algorithm remain in $\cK$ without using projection. 
The implicit regularization property has been investigated in 
nonconvex optimization and machine learning, including training  neural networks \cite{kubo2019implicit}, phase retrieval \cite{ma2017implicit}, and matrix completion \cite{chen2015fast,zheng2016convergence}. 
We emphasize that implicit regularization is a feature of both the \emph{problem} and the \emph{algorithm}. Next, we discuss two PO algorithms which are guaranteed to stay in the feasible set and achieve global convergence, thanks to the implicit regularization property.

\subsubsection{Algorithms}\label{subsec:alg_H2_Hinf}
For ease of exposition, we introduce the following notation
\small
\#\label{equ:def_mu_Ek}
&\tP_K:=P_K+P_KD(\gamma^2I-D^\tp P_K D)^{-1}D^\tp P_K,\quad E_K:=(R+B^\tp \tP_{K} B)K-B^\tp \tP_{K} A\\ 
&\Delta_K:=\sum_{t=0}^\infty \big[(I-\gamma^{-2} P_KDD^\tp)^{-\tp}(A-BK)\big]^tD(I-\gamma^{-2}D^\tp P_K D)^{-1}D^\tp\\ 
&\qquad\qquad\big[(A-BK)^\tp(I-\gamma^{-2} P_KDD^\tp)^{-1}\big]^t.\notag
\#
\normalsize 
Then we have the explicit gradient formula $\nabla J(K)=2((R+B^\tp\tP_K B)K-B^\tp \tP_K A)\Delta_K$, and can show the following two PO methods enjoy the implicit regularization property \cite{zhang2021policy}.  
\begin{flalign}
 {\rm \textbf{Natural Policy Gradient:}}~~~\qquad K^{n+1}&=K^n-\eta \nabla J(K^n)\Delta_{K^n}^{-1}=K^n-2\eta E_{K^n} &\label{eq:exact_npg}\\
 {\rm \textbf{Gauss-Newton:}}\qquad\qquad\quad\qquad\quad K^{n+1}&=K^n-\eta (R+B^\tp \tP_{K^n} B)^{-1}\nabla J(K^n)\Delta_{K^n}^{-1}\notag\\
 \qquad\qquad\qquad\qquad\qquad\qquad\qquad\quad &=K^n-2\eta (R+B^\tp \tP_{K^n} B)^{-1}E_{K^n},
 &\label{eq:exact_gn}
\end{flalign}
where $\eta>0$ is the stepsize. The updates resemble  the policy optimization updates for LQR as discussed in \S\ref{sec:LQR_results},  but with $P_K$ therein replaced by $\tP_K$. The natural PG update is related to the gradient over a Riemannian  manifold; 
while the Gauss-Newton update can be viewed as a special case of the quasi-Newton update.

\subsubsection{Global Convergence Guarantees}\label{subsec:H2_Hinf_theory}

The natural PG and Gauss-Newton updates in Equations  \ref{eq:exact_npg}-\ref{eq:exact_gn} enjoy the implicit regularization property formalized as below.

\begin{theorem}[Implicit Regularization]\label{thm:stability_update}
For any iterate $K=K^n\in\mathcal{K}$,
i.e., $\rho(A-BK)< 1$ and $\norm{(C-EK)(zI-A+BK)D}_\infty<\gamma$, 
	suppose that  the stepsize  $\eta$ satisfies
	\begin{itemize}
	\item Natural policy gradient in Equation \ref{eq:exact_npg}: $\eta\leq {1}/{(2\|R+B^\tp \tP_K B\|)}$,
	\item Gauss-Newton in Equation \ref{eq:exact_gn}: $\eta\leq {1}/{2}$.  
	\end{itemize}	 
	 Then the next iterate $K'=K^{n+1}$ obtained from Equations  \ref{eq:exact_npg}-\ref{eq:exact_gn} also lies in $\cK$.  
\end{theorem}

The proof of Theorem \ref{thm:stability_update} can be found in \cite[Section 5]{zhang2021policy}, which has a deep connection with Bounded Real Lemma \cite[Chapter $1$]{basar95}, \cite{zhou1996robust,rantzer1996kalman}.  
Theorem \ref{thm:stability_update} shows that the robustness of the controller is preserved when certain policy search directions are used. 
Intuitively, the natural PG and Gauss-Newton methods somehow exploit the information of $P_K$ to avoid the directions that may lead to infeasibility.
This implicit regularization property is due to a combination of a certain nonconvex objective (mixed $\cH_2/\cH_\infty$ control) and certain algorithms (natural PG and Gauss-Newton). 
With this property in hand, we are ready to state the global convergence result.
\begin{theorem}[Global Convergence \& Local Faster Rates]\label{theorem:global_exact_conv} Suppose that $K^0\in\cK$ 
 and $\|K^0\|<\infty$.  Then, under the 
stepsize choices\footnote{For natural PG, it suffices to require the stepsize $\eta\leq{1}/{(2\|R+B^\tp \tP_{K_0} B\|)}$ for the initial $K^0$.}  as in Theorem \ref{thm:stability_update},  updates in Equation~\ref{eq:exact_npg} and  Equation \ref{eq:exact_gn} both converge to the global optimum $K^*=(R+B^\tp \tP_{K^*} B)^{-1}B^\tp \tP_{K^*} A$,  
in the sense that 
$\sum_{n=1}^N \|E_{K^n}\|_F^2/N=O(1/N)$. Moreover, if $DD^\tp>0$, then  under the same 
stepsize choices,  
 both updates  converge to the optimal 
 $K^*$ with a \emph{locally linear} rate, i.e., the objective $\{J(K^n)\}$  converges to $J(K^*)$ with a linear rate. In addition, if $\eta=1/2$,  the Gauss-Newton update in Equation  \ref{eq:exact_gn} converges to   $K^*$ with  a locally \emph{Q-quadratic} rate. 
\end{theorem}

The proof of Theorem \ref{theorem:global_exact_conv} can be found in \cite[Section 5]{zhang2021policy}.  The theorem shows that though being nonconvex and non-coercive, certain PO methods can still find the globally optimal solution of mixed  $\cH_2/\cH_\infty$ control at  globally sublinear and locally superlinear rates. The global rate is sublinear, in contrast to the linear one for LQR, as the global  gradient dominance property  does not necessarily hold here. 

Finally, we remark that interestingly, \cite{zhang2021policy} also numerically compared the computation efficiency of PO methods and existing solvers for mixed  $\cH_2/\cH_\infty$ control (see, e.g., \cite{arzelier2011h2}). It has been shown that PO methods can indeed be much faster than these existing solvers 
(though note that these solvers can handle more general cases such as the output feedback case), especially for large-scale dynamical systems. 
This justifies the desired scalability of PO methods  for control synthesis; see \cite{zhang2021policy} for more numerical examples.

\subsubsection{Model-Free Implementations: Connections to Dynamic Games and Adversarial RL}\label{sec:connection_game_RARL}
There is a fundamental connection between mixed $\mathcal{H}_2/\mathcal{H}_\infty$ control and linear quadratic dynamic games \cite{basar95}.  This connection will not only allow us to implement the PO algorithms using model-free adversarial RL techniques, but also enable the development of PO methods for solving these dynamic  games.

As LQR can be viewed as the benchmark for single-agent RL in continuous space, linear quadratic (LQ) dynamic games serve as the  benchmark for studying multi-agent RL. Indeed, zero-sum LQ games have been investigated as 
a fundamental settings in multi-agent RL  \cite{al2007model,zhang2019policy,bu2020global,gravell2020policya,zhang2021provablya,zhang2021multi}. 
Specifically, consider  a zero-sum dynamic game with linear dynamics
$x_{t+1}=Ax_t+Bu_t+Dw_t$. The objective of  player $1$ (player $2$) is to minimize (maximize) the value function $
\cC:=\EE_{x_0\sim\cD}\big[\sum_{t=0}^\infty (x_t^\tp Qx_t+u_t^\tp R^uu_t-w_t^\tp R^w w_t)\big], $ 
where $x_0\sim \cD$ for some distribution $\cD$, and $(Q, R^u, R^w)$ are positive definite matrices. It is known that the Nash equilibrium (NE), the solution concept for the problem, can be achieved  
with state-feedback policy classes, i.e., there exists a pair $(K^*,L^*)$, such that the NE satisfies $u_t^*=-K^*x_t$ and $v_t^*=-L^*x_t$ \cite{basar95,basar99}. Hence, one can parameterize the controllers using  matrices $(K,L)$, and solve for $\min_K\max_L\cC(K,L)$, where $\cC$ is the accumulated cost under this pair $(K,L)$. This leads to a multi-agent PO problem. 
Intriguingly,  the NE to the game is provided by the solution to a specific mixed $\cH_2/\cH_\infty$ control problem \cite{basar95}.  With this connection, the natural PG and Gauss-Newton methods in Equation \ref{eq:exact_npg} and Equation  \ref{eq:exact_gn} can be equivalently transformed into provably convergent double-loop PO algorithms for the above LQ game. Related algorithmic developments have been documented in \cite{zhang2019policy,bu2019global}. The game formulation for mixed $\mathcal{H}_2/\mathcal{H}_\infty$ control is powerful in that these double-loop variants of the natural PG and Gauss-Newton methods  can be implemented in a model-free manner. See   \cite{zhang2019policy,keivan2021model,zhang2021derivative} for detailed discussions on model-free implementations and related sample complexity results.

An important issue in RL is the simulation-to-real gap. One common remedy is to use robust adversarial RL (RARL) algorithms which jointly learn  a  \emph{protagonist} 
and an \emph{adversary}, where the former learns to robustly perform the control tasks under the disturbances created by the  latter \cite{morimoto2005robust,pinto2017robust}. 
Policy-based RARL methods can be viewed as model-free variants of multi-agent PO methods for dynamic games. Therefore,  the PO theory for mixed $\mathcal{H}_2/\mathcal{H}_\infty$ control can also be applied to study the properties of RARL algorithms in the LQ setting. 
See \cite{zhang2020stability} for more details about this connection and results.

\vspace{5pt}\noindent\textbf{Indefinite LQR.} 
The inner-loop subroutine in both zero-sum dynamic games and RARL reduces to a generalized LQR problem whose state cost matrix $Q$ is not positive semi-definite. 
The formulation of indefinite LQR is similar to that in \S\ref{sec:formulation}, except that the $Q$ and $R$ matrices are symmetric but indefinite. 
Then the  cost may not be coercive, and the descent of the cost does not ensure the stability of the iterates \cite{bu2020global}. Nevertheless, global convergence of policy optimization methods can be established; see \cite{bu2020global}.

\subsection{PO for  $\mathcal{H}_\infty$  State-Feedback Synthesis: Nonsmoothness and Convergence}\label{sec:Hinf_optimal}

In this section, we consider the state-feedback $\mathcal{H}_\infty$ optimal control  problem. 
 Classical convex approaches for this task requires
reparameterizing the problem into a higher-dimensional convex domain~\cite{zhou1996robust,Dullerud99,befb94}.
In contrast, we view this problem as a benchmark of PO for robust control. 
We will discuss how to provably find  the optimal $\mathcal{H}_\infty$ controller in the policy space directly. 
Recall that for the PO formulation of $\mathcal{H}_\infty$ state-feedback synthesis, the cost function $J(K)$ is given by Equation \ref{eq:hinfcost}, and the feasible set $\mathcal{K}$ is specified by Equation~\ref{equ:def_stabilizing_K}. 
A main technical challenge here is that this $\mathcal{H}_\infty$ cost  can be non-differentiable at  some important feasible points, e.g., the optimal points  \cite{apkarian2006nonsmooth,arzelier2011h2,gumussoy2009multiobjective,burke2020gradient}. 
From Equation~\ref{eq:hinfcost}, we can see that this cost function is subject to two sources of nonsmoothness.  Namely, the largest eigenvalue for a fixed frequency $\omega$ is nonsmooth, and the optimization step over $\omega\in [0, 2\pi]$ is also nonsmooth. 
We also know that the feasible set from Equation \ref{equ:def_stabilizing_K} is nonconvex. 
Hence, the resultant PO problem for $\mathcal{H}_\infty$ state-feedback synthesis is nonconvex nonsmooth. There has been a large family of nonsmooth $\mathcal{H}_\infty$ policy search algorithms developed based on the concept of Clarke subdifferential \cite{apkarian2006nonsmooth,arzelier2011h2,gumussoy2009multiobjective,burke2020gradient}. 
However, the global convergence theory of PO methods on the $\mathcal{H}_\infty$ state-feedback synthesis has not been established until very recently.
 Next, we review such global convergence results from~\cite{Guo2022hinf}.

First, we introduce a few concepts related to subdifferential of nonconvex functions.
A function $J:\mathcal{K}\rightarrow \R$ is locally Lipschitz if for
any bounded $S\subset \mathcal{K}$, there exists a constant $L > 0$ such that $|J(K)-J(K')|\le L \norm{K-K'}_F$ for all $K,K'\in S$. 
Based on Rademacher's theorem, a locally Lipschitz function is differentiable almost everywhere, and the Clarke subdifferential is well defined for all feasible points. We define the Clarke subdifferential as $\partial_C J(K):=\conv\{\lim_{i\rightarrow \infty}\nabla J(K_i):K_i\rightarrow K,\,K_i\in\dom(\nabla J)\subset \mathcal{K}\}$,
where $\conv$ denotes the convex hull.  For any given direction $V$ (which has the same dimension as $K$), the generalized Clarke directional derivative of $J$ is defined as
\begin{align}
J^{\circ}(K,V):=\lim_{K'\rightarrow K}\sup_{t\searrow 0} \frac{J(K'+tV)-J(K')}{t}.
\end{align}
In contrast, the (ordinary) directional derivative is defined as follows (when existing)
\begin{align}
J'(K,V):=\lim_{t\searrow 0} \frac{J(K+tV)-J(K)}{t}.
\end{align}

In general, the Clarke directional derivative can be different from the (ordinary) directional derivative which may not even exist for some feasible points.
The objective function $J(K)$ is subdifferentially regular if for every $K\in \mathcal{K}$, the ordinary directional
derivative always exists and coincides with the generalized one for every direction, i.e.,  $J'(K,V)=J^{\circ}(K,V)$. The following result holds for the $\mathcal{H}_\infty$ objective function.
\begin{proposition}
Let $\mathcal{K}$ be non-empty. Then the $\mathcal{H}_\infty$ objective function defined by Equation~\ref{eq:hinfcost} is
locally Lipschitz and subdifferentially regular over the stabilizing feasible set~$\mathcal{K}$.
\end{proposition}

The above result is well known. See \cite{Guo2022hinf} for more explanations.
Consequently,  the Clarke subdifferential for the $\mathcal{H}_\infty$ objective function is well defined for all $K\in \mathcal{K}$. We say that $K^\dag$ is a Clarke stationary point if  $0\in \partial_C J(K^\dag)$. The subdifferentially regular property guarantees that the directional derivatives at any Clarke stationary points $J'(K^\dag, V)$ are always non-negative.
Since $\mathcal{K}$ is open, the global minimum has to be a Clarke stationary point. Searching Clarke stationary points provably requires advanced subgradient algorithms, since
generating a good descent direction for nonsmooth optimization is non-trivial. The concept of Goldstein subdifferential \cite{goldstein1977optimization} is relevant and stated below.
\begin{definition}
Suppose $J$ is locally Lipschitz. Given a point $K\in\mathcal{K}$ and a parameter $\delta>0$, the Goldstein subdifferential of $J$ at $K$ is defined to be the following set 
\begin{align} \label{Gold_sub}
\partial_\delta J(K):=\conv \left\{\cup_{K'\in\mathbb{B}_\delta(K)} \partial_C J(K')\right\},
\end{align}
where $\mathbb{B}_\delta(K)$ denotes the $\delta$-ball around $K$. It is implicitly assumed $\mathbb{B}_\delta (K)\subset\mathcal{K}$.
\end{definition}
Importantly, the minimal norm element of the Goldstein subdifferential generates a good descent direction. The minimal norm element in $\partial_\delta J(K)$, denoted as $F$, will satisfy 
$J(K-\delta F/\norm{F}_F)\le J(K)-\delta \norm{F}_F$, if we have $\partial_\delta J(K)\subset \mathcal{K}$.
This fact has inspired the developments of Goldstein's subgradient method \cite{goldstein1977optimization} and related variants for nonsmooth $\mathcal{H}_\infty$ control \cite{arzelier2011h2,gumussoy2009multiobjective,burke2020gradient}. Recently, it has been proved  that Goldstein's subgradient method  can be guaranteed to find the global minimum of the $\mathcal{H}_\infty$ state-feedback synthesis problem despite the nonconvexity of the feasible set \cite{Guo2022hinf}. 
We summarize this result as~follows.
\begin{theorem}
Suppose $(Q,R)$ are positive definite, and the pair $(A,B)$ is stabilizable. Denote $J^*=\min_{K\in\mathcal{K}} J(K)$.  For $\mathcal{H}_\infty$ state-feedback synthesis, the following statements hold.
\begin{enumerate}
    \item The $\mathcal{H}_\infty$ objective function defined by Equation \ref{eq:hinfcost} is coercive over $\mathcal{K}$.
    \item For any $K\in\mathcal{K}$ satisfying $J(K)>J^*$, there exists $V\neq 0$ such that $J'(K,V)<0$.
\item Any Clarke stationary points of  the $\mathcal{H}_\infty$ objective function are global minimum.
\item The sublevel set $\mathcal{K}_\gamma$ is compact. There is a strict separation between  $\mathcal{K}_\gamma$ and $\mathcal{K}^c$ (which is the complement of the feasible set $\mathcal{K}$). In other words, we have ${\rm dist}(\mathcal{K}_\gamma, \mathcal{K}^c)>0$.
\item Suppose $K^0\in \mathcal{K}$. Denote $\Delta_0:={\rm dist}(\mathcal{K}_{J(K^0)}, \mathcal{K}^c)>0$. Choose $\delta^n=\frac{0.99\Delta_0}{n+1}$ for all $n$.
Then Goldstein's subgradient method $K^{n+1}=K^n-\delta^n F^n/\norm{F^n}_F$ with $F^n$ being the minimum norm element of $\partial_{\delta^n} J(K^n)$ is guaranteed to stay in $\mathcal{K}$ for all $n$. In addition,  we have $J(K^n)\rightarrow J^*$ as $n\rightarrow \infty$.
\end{enumerate}
\end{theorem}
The coerciveness can be proved using the positive definitness of $(Q,R)$.  Statement~2 follows from the convex paramterization for $\mathcal{H}_\infty$ state-feedback synthesis, and we will discuss more about this point later. Statement 3 can be proved by combining Statement 2 and the sudifferential regular property.  Statement 4 is a consequence of Statement 1. Then one can combine the descent property of Goldstein's subgradient method and Statement 4 to prove the convergence result in Statement 5. Due to some subtlety of nonsmooth nonconvex optimization, sample complexity of PO on nonsmooth $\mathcal{H}_\infty$ synthesis remains unknown.
Discussions on model-free implementations and related issues can be found in~\cite{Guo2022hinf}.

\vspace{0.1in}
\noindent
\textbf{Related work in other settings.}
PO has also been investigated in other control problems  with robustness and risk-sensitivity concerns, other than the LEQG/$\mathcal{H}_\infty$ settings discussed in this survey. See \cite{turchetta2020robust,gravell2020learning,gravell2020policya,pang2021robust,pang2021robusta,venkataraman2019recovering,zhao2021infinite,zhao2021primal,zhao2021global,keivan2021model}.

\section{Case III: PO  with Partial Observations} \label{advance:partial-observations}

In this section, we examine the more challenging case 
of control with partial observation. 
When the system's state is not directly measured, there is an
intricate balance between the achievable control performance and the class of controllers used in PO.
Depending on the policy parameterization, the optimization landscape can become quite different.
We will first survey several recent results on the optimization landscape of PO for LQG, and then point out some of the subtle aspects for more general  output feedback and structured synthesis problems.

\subsection{PO for Linear Quadratic Gaussian Control: Optimization Landscape} \label{sec:LQG_opt_landscape}

In order to characterize the performance of PO algorithms such as PG methods for LQG, it is necessary to understand the landscape of the associated PO formulation in Equation~\ref{eq:opt}, with the cost function given in Equation \ref{eq:LQG_cost_formulation_discrete} and the feasible set given in Equation~\ref{eq:LQG-constraint}. Following standard setup in the literature, we assume that the pairs $(A,B)$ and $(A, W^{1/2})$ are controllable, and $(C,A)$ and $(Q^{1/2}, A)$ are observable. It has been shown that with these
assumptions, the set of stabilizing controllers $\mathcal{K}$ is non-empty, open, unbounded, and can be nonconvex. Moreover, the cost function $J(K)$ is real analytic on the underlying set~$\mathcal{K}$.

 However, beyond these properties, until recently little was known about the geometric and analytical properties of the PO formulation of
 LQG. We will mainly summarize results on the optimization landscape of LQG from ~\cite{zheng2021analysis}, especially with respect to the connectivity of the stabilizing set $\mathcal{K}$ and the structure of the stationary points. 
Several related extensions can be found in \cite{duan2021optimization,duan2022optimization,mohammadi2021lack,hu2022connect}. Before introducing the results, we discuss a special structure for LQG with state-space dynamical controller parameterization in Equation \ref{eq:LQG-K}. It is known that the optimal feedback controller is unique in the frequency domain~\cite[Theorem 14.7]{zhou1996robust}. However, in time domain, this
 controller is not unique: 
 consider the similarity transformation for the state-space form of the controller and note that 
 the two controller parameterizations, 
\[
K=\begin{bmatrix}
0 & C_{K} \\ B_{K} & A_{K}
\end{bmatrix} \quad \mbox{and} \quad \mathcal{T}(T,K):=\begin{bmatrix}
0 & C_{K}T^{-1} \\ TB_{K} & \,T\!A_{K}T^{-1}
\end{bmatrix},
\]
where $T$ is an invertible matrix, 
have identical input-output behavior regardless of the choice of $T$. Thus the cost is invariant with respect to this similarity transformation. Besides this invariance, when a controller $K$ is non-minimal, i.e., $(A_K,B_K)$ is not controllable or $(A_K, C_K)$ is not observable, one can use model reduction to remove the uncontrollable/unobservable modes while keeping the cost the same. 
We will now summarize the main results from \cite{zheng2021analysis}.
First, we have the following theorem on the connectivity of $\mathcal{K}$. \begin{theorem} \label{Theo:disconnectivity}
The set ${\mathcal{K}}$ has at most two path-connected components. When  ${\mathcal{K}}$ has two connected components, these components are diffeomorphic under the similarity transformation $\mathcal{T}(T,\cdot)$ for any invertible matrix with $\det T<0$.
\end{theorem}
On a conceptual level, the proof is based on a convex reparameterization of the LQG problem. Figure~\ref{fig:LQG-connectivity} shows two examples for $\mathcal{K}$, with one or two connected components. 
\begin{figure*}[!t]
	\centering
	\begin{tabular}{cc}
		\hskip-110pt\includegraphics[width=0.3\textwidth]{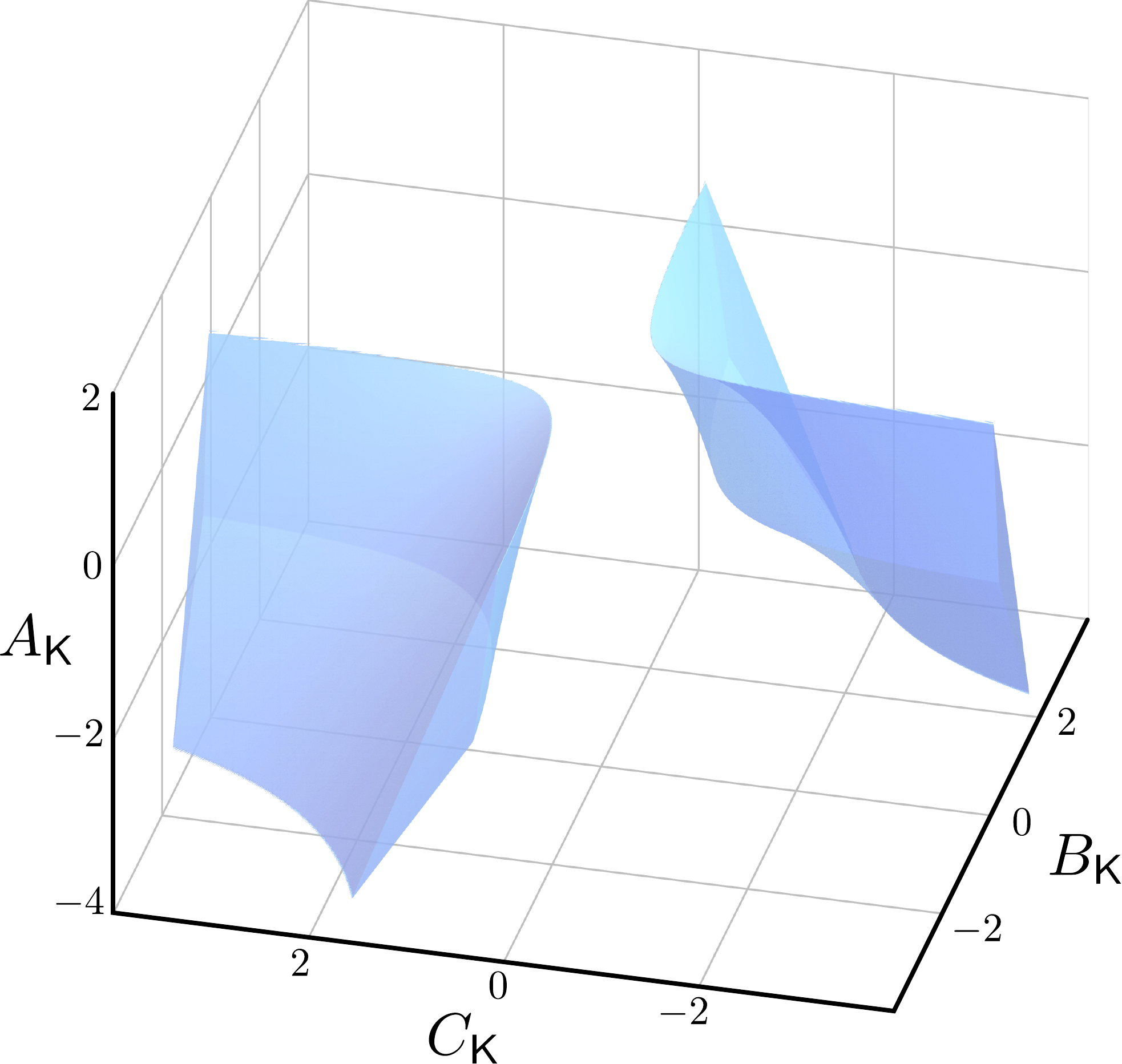}
		&
		\hskip-190pt\includegraphics[width=0.3\textwidth]{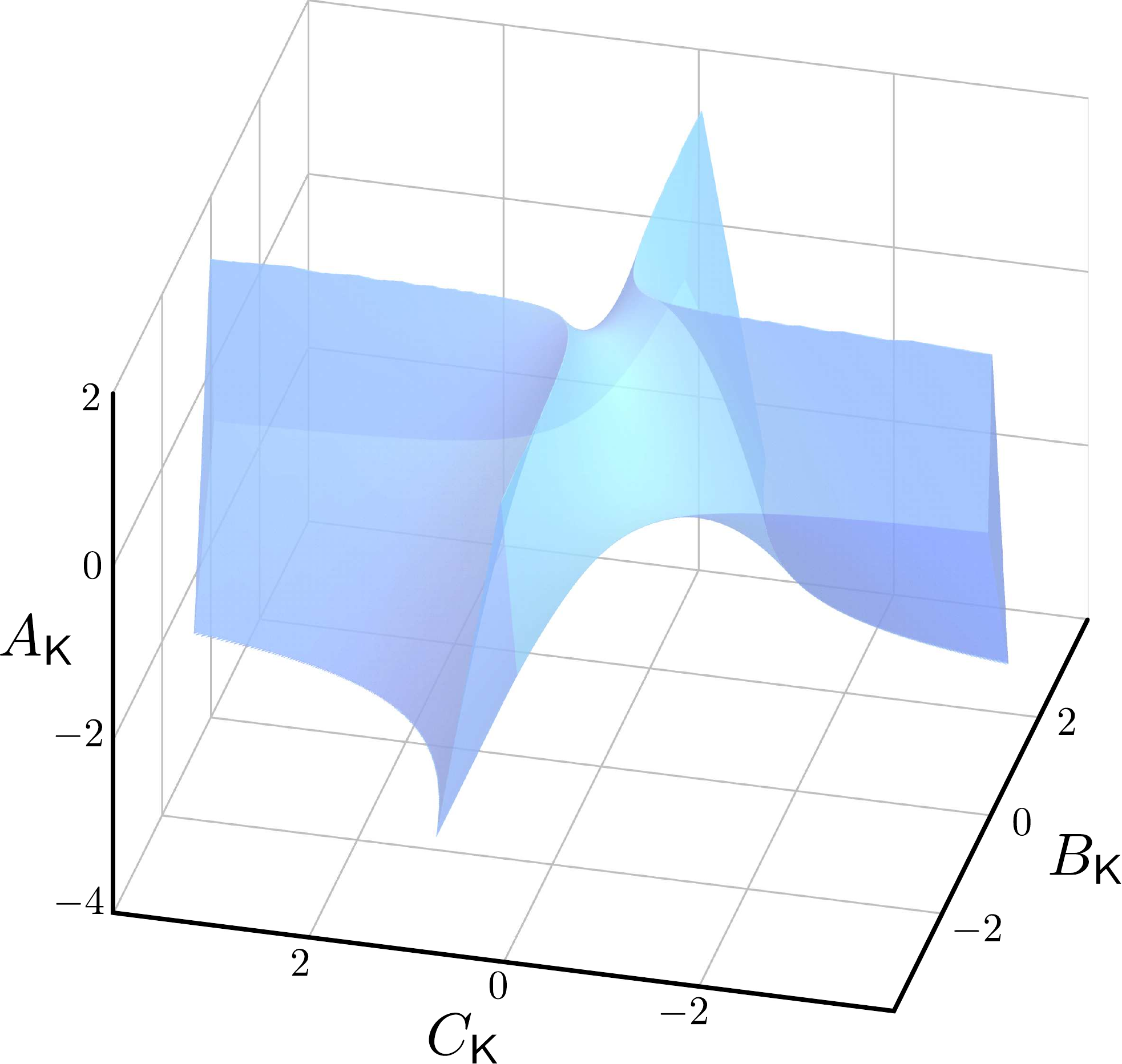}
		\\ 
		\hskip-95pt(a) System   parameters: $A=\frac{3}{2}, B=1, C=1$  & \hskip -190pt(b) System parameters: $A=\frac{2}{3}, B=1, C=1$
	\end{tabular}
	\vspace{7pt}
	\caption{Two examples of the feasible set $\mathcal{K}$ for LQG: (a) a disconnected feasible set, (b) a connected feasible set.}
	\label{fig:LQG-connectivity}
\end{figure*}

For policy gradient algorithms and other local search methods, the connectivity of the domain (the set of stabilizing controllers) is important since there are
no jumping iterates between different connected components. 
Nevertheless, in light of Theorem~\ref{Theo:disconnectivity} and 
the fact that the similarity transformation does not change the input/output behavior of a controller, it makes no difference to search over either path-connected component in $\mathcal{K}$ even if $\mathcal{K}$ is not path-connected.
In fact, one can further show that any strict sublevel sets of the LQG PO problem have very similar connectivity properties \cite{hu2022connect}.
Such observations
are encouraging for devising gradient-based local search algorithms for LQG.

\begin{figure*}[!t]
	\centering
	\begin{tabular}{cc}
		\hskip-110pt\includegraphics[width=0.3\textwidth]{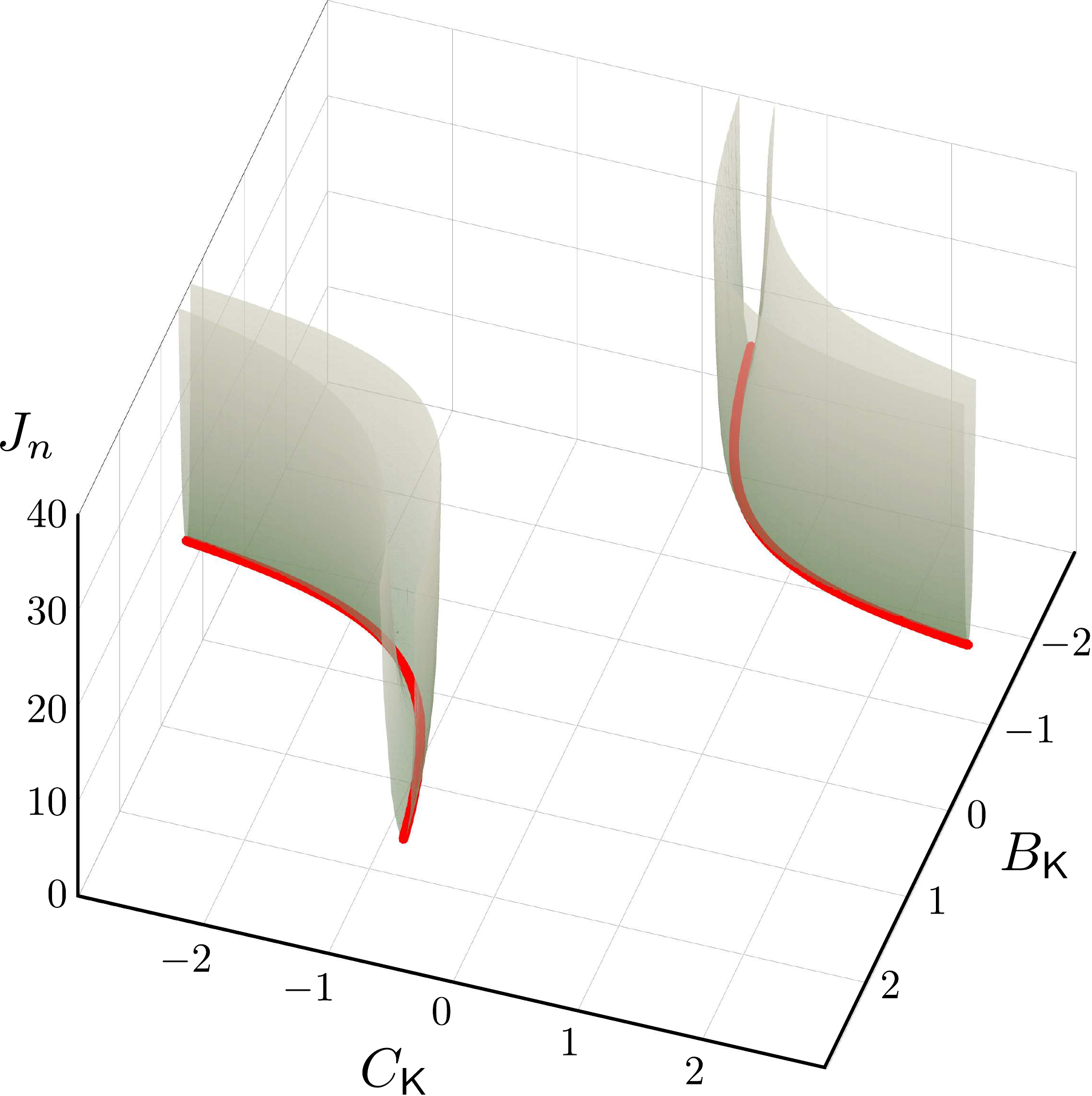}
		&
		\hskip-190pt\includegraphics[width=0.3\textwidth]{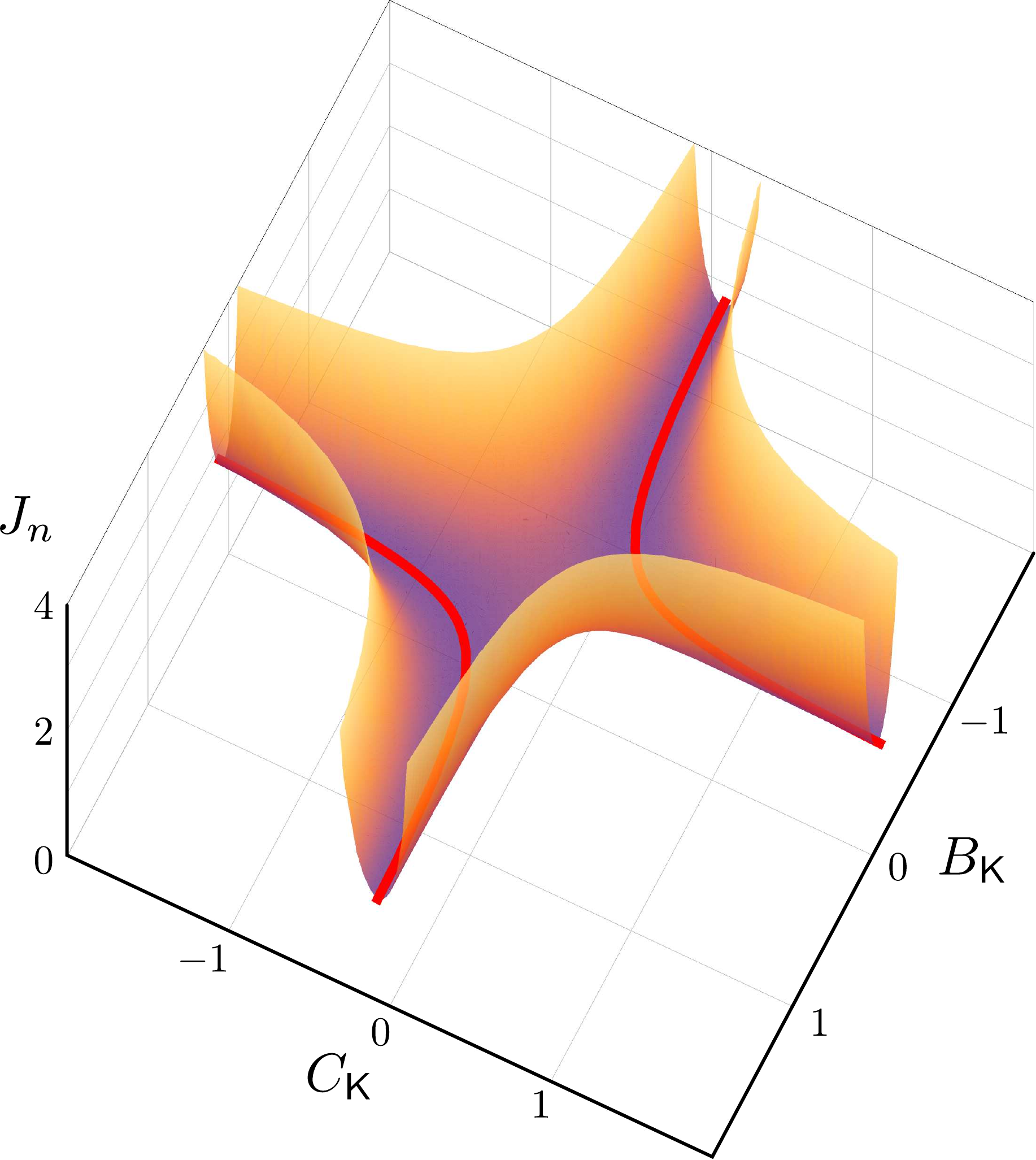}
		\\
		\hskip-95pt(a) System  parameters: $A=\frac{3}{2}, B=1, C=1$  & \hskip -190pt(b) System parameters: $A=\frac{2}{3}, B=1, C=1$ 
	\end{tabular}
	\vspace{8pt}
	    \caption{Non-isolated and disconnected globally optimal LQG controllers. In both cases, we set $Q = 1, R = 1, V = 1, W = 1 $. (a) LQG cost for the system in Figure~\ref{fig:LQG-connectivity}.a when fixing $A_{K} =-0.07318$, for which the set of globally optimal points $\left\{(B_{K}, C_{K}) \mid B_{K}C_{K}=-0.13684\right\}$ has two connected components. (b) LQG cost for the system in Figure~\ref{fig:LQG-connectivity}.b when fixing $A_{K} = -0.67360$, for which the set of globally optimal points $\left\{(B_{K}, C_{K}) \mid B_{K}C_{K}=-1.18113 \right\}$ has two connected components.}
	    \label{fig:LQG-J}
\end{figure*}

Though the  LQG problem has a nice property in terms of the connectivity of $\mathcal{K}$ (given in Equation \ref{eq:LQG-constraint}), its 
optimization landscape is otherwise more complicated that cases we saw earlier.  Figure~\ref{fig:LQG-J} shows two examples for the LQG cost $J(K)$.  Firstly, it is easy to see from the figure that LQG has non-unique and non-isolated global optima in the state-space form. This feature often adds difficulty in establishing convergence of policy optimization methods. Secondly, it is also straightforward to show that the LQG cost is non-coercive. One way to see this is to consider the similarity transformation using $\lambda I$. By letting $\lambda\rightarrow \infty$, we see that $|B_K|\rightarrow \infty$ but the cost remains  unchanged. Non-coerciveness makes it challenging to establish convergence even to stationary points. Lastly, 
LQG could have saddle points that are not optimal. Moreover, if $K$ is a stationary point, then all of its ``similar'' controllers $\mathcal{T}(T,K)$ are also stationary points for any nonsingular~$T$. Also if $K\in \mathcal{K}$ is  a non-minimal stationary point, its minimal reduction could generate an infinite number of more stationary points.  Consequently, it is nontrivial to find an optimal controller, or even certify an optimal controller, through policy gradient methods. Nevertheless, there is one clean case for the certification of the globally optimal points.   
\begin{theorem} \label{theo:stationary_points_globally_optimal_discrete}
 All {\it minimal} stationary points\footnote{Here ``minimal" refers to being both observable and controllable as a dynamical system.}  
  $K\in\mathcal{K}$ in the LQG problem in Equation~\ref{eq:LQG_cost_formulation_discrete} are globally optimal, and they are related to each other by a similarity transform. 
\end{theorem}

The above results indicates that when running the policy gradient methods, if the iterates converge to a minimal stationary point, 
a globally optimal controller has been found. However, if the stationary points are non-minimal, one cannot say much about its optimality. Indeed, \cite{zheng2021analysis} provides examples showing the existence of \textit{saddle} points of LQG. Though there are recent developments in perturbed gradient methods that can be ensured to escape strict saddle points (saddle points with indefinite Hessian; i.e., escape direction exists in the second-order), there exist LQG instances with saddle points whose Hessian is degenerate. These observations pose challenges in analyzing the performance of policy gradient methods applied to LQG. 

The recent work \cite{zheng2022escaping} introduced a novel perturbed policy gradient (PGD) method that is capable of escaping various bad stationary points (including high-order saddles). Based on the specific structure of LQG, this paper uses a reparameterization procedure which converts the iterate from a high-order saddle to a strict saddle, from which standard randomly perturbed PGD can escape efficiently. It also characterizes the high-order saddles that can be escaped by the proposed algorithm; 
however, there is still a lack of an end-to-end theorem to characterize the iteration complexity of the algorithm. It remains an open question to analyze the performance of policy gradient methods on LQG by 1) establishing conditions under which the algorithms will converge to, at least, stationary points, 2) designing effective ways to escape non-optimal stationary points, at least saddle points, 3) characterizing the algorithm complexity of the designed algorithms, and 4) developing sample-based methods and analyzing the sample complexity. It is also worth mentioning that the global convergence of PO for a simpler estimation problem has been proved in~\cite{umenberger2022globally}. See Section \ref{sec:convex_para_grad_dom} for more discussions.

\subsection{Output Feedback and Structured Control}

We now shift our attention to another class of synthesis 
problems with partial observations, namely, output feedback and structured control. 
First, we like to point out that policy synthesis on partially available data (not necessary the underlying state) is of great interest in applications, particularly for large scale systems. For example, in decentralized control, stabilizing feedback with a particular sparsity pattern is desired; in the output feedback case, one aims to design a stabilizing policy that can be factored with its right multiplicand as the observation map. These problems
 can conveniently be formalized
in form of Equation~\ref{eq:opt}, 
where ${\cal K}$ becomes a subset of stabilizing (static or dynamic) feedback policies: for both output feedback and structured
synthesis, ${\cal K}$ is a linearly constrained subset of stabilizing feedback gains.
The 
PO perspective adopted in this survey then immediately offers an algorithm for 
these problems, namely, a projected first order update\footnote{The projection is used to enforce sparsity/structure patterns. The projection does not involve stability/robustness concerns.}. A natural 
question is whether such an intuitive generalization has any theoretical guarantees of convergence; the short answer, however, is negative. We now summarize some of our current understanding
of why this is the case, intermingled with some more encouraging results.

A major obstacle in 
guaranteeing  convergence to
the global optimum 
is due to the geometry of the 
corresponding set ${\cal K}$---
and not only its non-convexity inherited from the set of stablizing controllers. Rather,
due to the intricate geometry of this set, its intersection with linear subspaces
can result in disconnected sets; an analogous phenomena in the case of LQG was examined in the previous section. This is a known fact from classic control
in the context of output feedback and the method of root locus, where the (scalar)
feedback gain can undergo intervals of being stabilizing or not; an example is shown 
in Figure~\ref{fig:outputfeedback}(a).
 \begin{figure*}[!t]
\centering
\begin{tabular}{ccc}
\hspace{-145pt}	\includegraphics[width=0.28\textwidth]{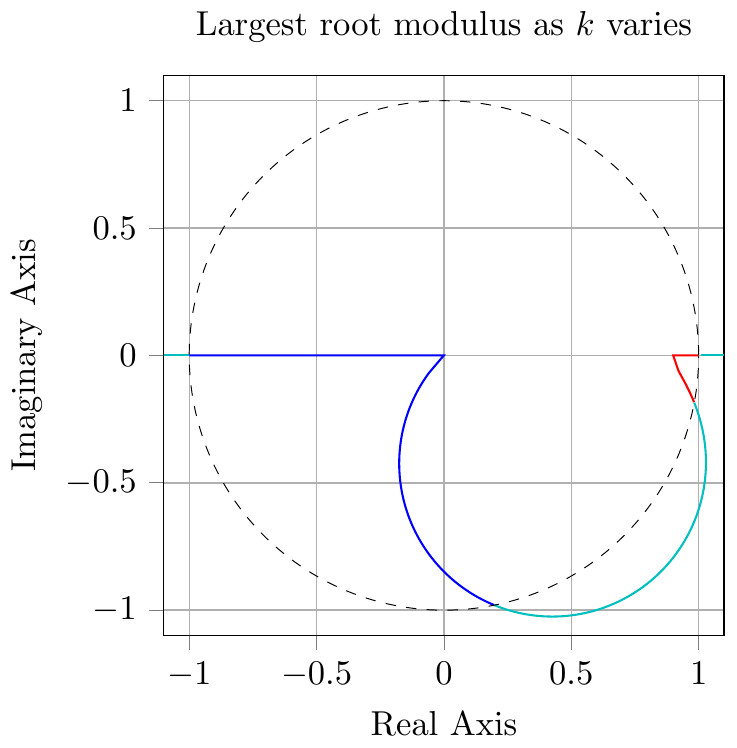} &
\hspace{-260pt} 
\includegraphics[width=0.4\textwidth]{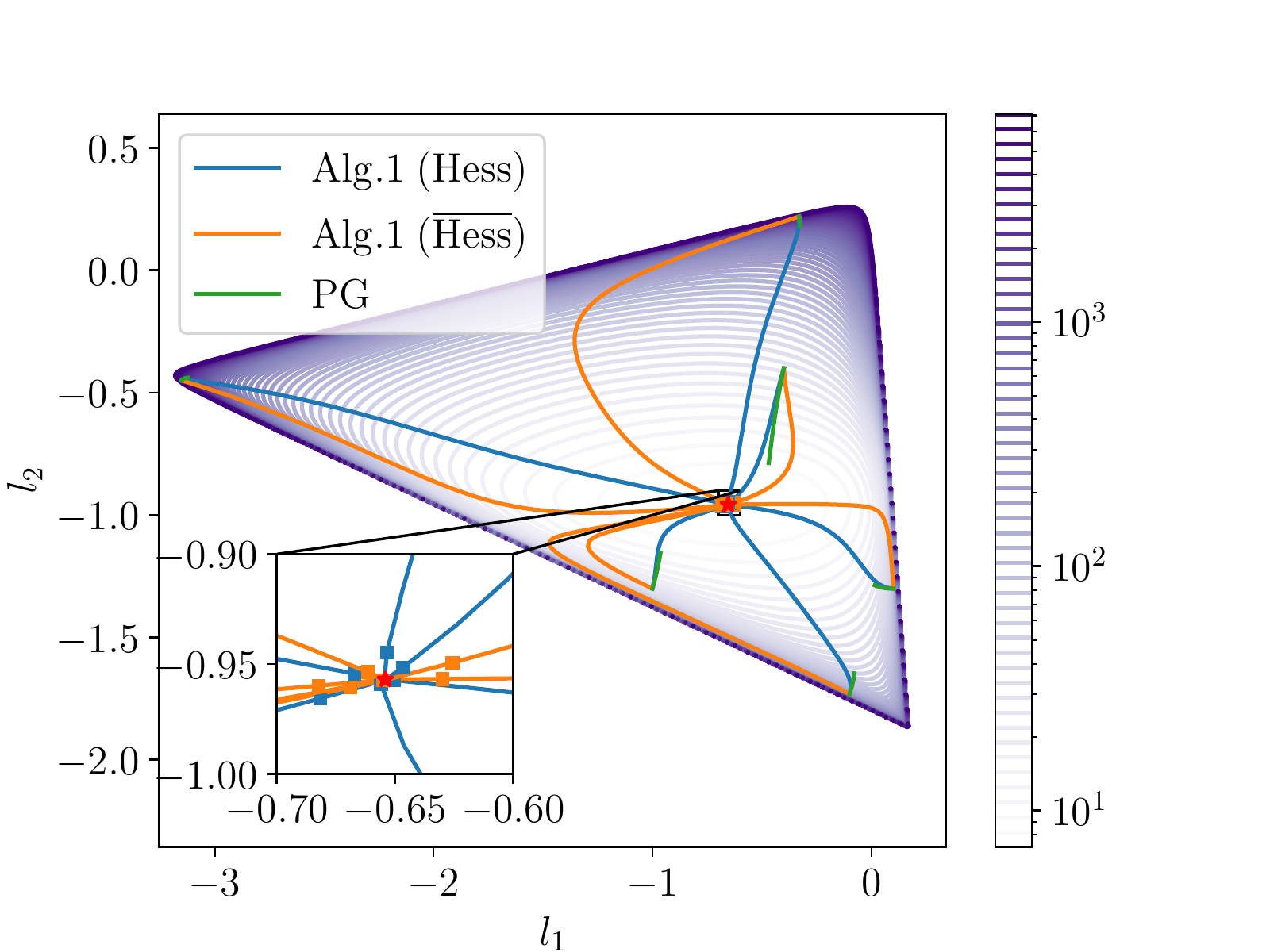} &  
\hspace{-265pt}\includegraphics[width=0.36\textwidth]{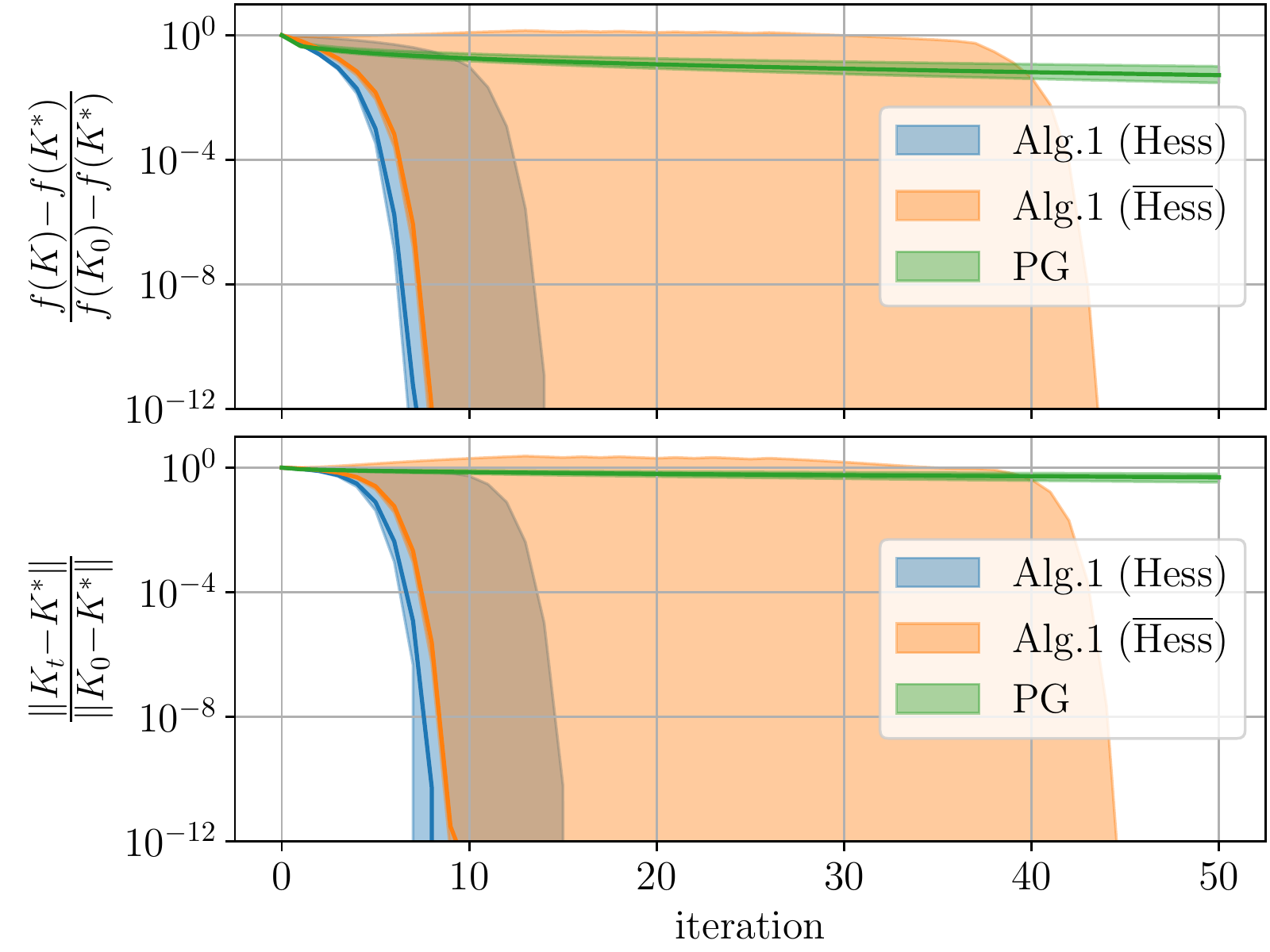}\\
\hspace{-130pt}(a) & \hspace{-275pt} (b) & \hspace{-255pt} (c)
\end{tabular}
\vspace{10pt}
\caption{(a) The largest modulus root of (discrete-time) feedback system with blue and the red segment corresponding to the stabilizing intervals~\cite{bu2019topological},  (b) The geometry of output feedback synthesis with respect to three different classes of algorithms (PG: projected gradient; ${\mbox{Hess}}$: scaled Newton update with Hessian using the  Riemmannian geometry of ${\cal K}$; $\overline{\mbox{Hess}}$: scaled Newton update with Hessian using the Euclidean  geometry of ${\cal K}$ ~\cite{Talebi2022-on}, (c) Performance of the aforementioned 
policy gradient descent for output feedback synthesis; similar results have been reported for structured synthesis in~\cite{Talebi2022-on}.}
	\label{fig:outputfeedback}
\end{figure*}
However, it is surprising that the number of such connected components even for SISO (single-input single-output) systems was not explicitly characterized until recently. A policy optimization perspective on control synthesis only makes this observation more compelling.
\begin{theorem}
The set of stabilizing output feedback gains for an $n$-dimensional SISO system, when nonempty, has at most $\lceil{n/2}\rceil$ connected components.
\end{theorem}
The proof of this result, as well as its analogue for continuous time case, and 
precise characterization of these intervals can be found in~\cite{bu2019topological}. 
Less is known about the number of connected components for MIMO (multi-input multi-output) output feedback, knowledge of which could be useful for initializing PO algorithms. Nevertheless, some topological characterizations of these sets, including sufficient conditions for connectedness of structured stabilizing gains have been obtained in the literature~\cite{bu2019topological,feng2019exponential,ding2019aggressive}.

The above topological insights are important in the context of policy search updates when they are required to stay stabilizing. As such, for general structured (including output feedback) synthesis, it is judicious to examine instead convergence to local optima or stationary points. The first observation in this direction has been made in~\cite{bu2019lqr}, where it was shown that projected gradient descent has a sublinear convergence to the (first-order) stationary point of structured LQR synthesis; analogous results for output
feedback synthesis have also been reported in~\cite{fatkhullin2020optimizing}.
Sample-based learning methods for reaching a first-order stationary point using zeroth-order methods for structured synthesis were examined in~\cite{li2021distributed}.
The sublinear convergence to a stationary point--the moment we impose a linear constraint on the set of stabilizing feedback gains--
only hints at the fact that we are not fully utilizing the underlying geometry of the feedback synthesis problem. 
These observations have motivated a new line of work on structured synthesis, that is also inline with the natural gradient iteration and quasi-Newton method for LQR presented earlier. 
The key missing insight pertains of course to the Hessian and the Riemannian geometry of
the set of stabilizing feedback gains as well as their linearly constrained subsets. As it turns out, PO for LQR is closely related to iterative
approaches for solving the Riccati equation that have subsequently
been adopted for data-driven setups. In particular, 
it can be shown that what is known as Hewer algorithm
for LQR~\cite{hewer1971iterative}, is really a realization 
of a quasi-Netwon update for a particular choice of a stepsize.
In our desire to understand the fundamental limitation
of PO for control synthesis, it is thus relevant to
characterize the {\em Newton update} 
on the set of stabilizing feedback gains as well as its linearly constrained subset.
This more geometric
question is still relevant in the context of first order methods,
as it provides fundamental insights into how to recover a linear, or
even quadratic rate of guaranteed convergence to stationary or even locally optimal points for policy optimization methods on structured synthesis. These topics have been thoroughly analyzed
in~\cite{Talebi2022-on}, including ``proper'' construction of the Hessian for
the LQR problem through the intrinsic Riemannian
connection,
how to extend this Hessian to linearly constrained
subsets ${\cal K}$, and how to choose the stepsize in the corresponding
iterates to remain stabilizing. The upshot of this analysis are results such as the following.
\begin{theorem}
Suppose that $K^*$ is a nondegenerate local minimum of the LQR on the
linear constrained subset ${\cal K}$. Then there
is a neighborhood around $K^*$ and a positive scaling, for which the scaled
Riemannian Netwon policy update remains stabilizing and converges
to $K^*$ at a linear--and eventually--at a quadratic rate.
\end{theorem}
A representative scenario is shown in Figure~\ref{fig:outputfeedback}(c) for the output feedback
problem, where ``$\mbox{Hess}$'' and ``$\overline{\mbox{Hess}}$'' refer to the Newton steps using the Hessian obtained from the Riemannian and Euclidean connections,
respectively, and the PG refers to the projected gradient descent algorithm. Ensuring stability during the course of these iterates, particular to control synthesis, often makes the analysis of these algorithms more intricate.

\section{The Role of Convex Parameterization}\label{sec:convex_para_grad_dom}  

There is a large body of literature on the reparameterization of various control problems 
to represent them as convex problems 
\cite{befb94,Scherer2004}. In this section, we discuss the connections between such convex approaches and PO, showing LMI formulations for control design lead to desired landsacpe properties for PO. 

We have seen successful application of PO to a range of control problems in the previous sections, in many cases achieving the \emph{globally optimal} policy. A natural question is whether there is a unified approach to determining when stationary points for PO are global minima.
In this section, we revisit the gradient dominance property given in Definition~\ref{def_coer},
and 
provide a unified framework to show some related inequality holds for a large family of PO problems, despite the nonconvexity of the cost $J(K)$ as a function of $K$. 
This viewpoint gives insights into the ``mysterious" emergence of gradient dominance (or the PL property) in various control problems that are nonconvex in $K$, providing a general tool to determine when stationary points for nonconvex PO problems are actually global minimum.

Intuitively, the gradient dominance property implies that $J(K)$ is close to the optimal value for any $K$ with small gradient norm, from which one can directly conclude nice optimization landscape/convergence properties\footnote{Convergence rates will depend on the values of the degree $p$, see \cite{LiPong2018KLexponent} for more properties. In particular $p=1$ gives a sublinear convergence rate, and $p=2$ gives a linear rate. }.
Our goal in this section is to show how to use the existence of convex parameterizations, together with important additional assumptions on the \emph{map} between the variables in the nonconvex and convex problems, to conclude such a desired property or some closely-related variant for the nonconvex $J(K)$.

We begin by considering an abstract description 
of the following pair of problems 
\begin{align}\label{f_K}
\min_{K}&\quad J(K),\quad
\mbox{s.t.}\ K\in\mathcal{K},
\end{align}
\begin{align}\label{lmi}
\min_{L,P,Z}&\quad f(L,P,Z),\quad 
\mbox{s.t.}\quad (L,P,Z)\in\mathcal{S}, 
\end{align}
where $\mathcal{K}$ describes the set of desired controllers (typically, stablizing set), and $\mathcal{S}$ captures the appropriate  constraint sets (typically LMIs), which are determined for each problem case; see examples below.
The following key assumption on the pair of problems in Equation \ref{f_K} and Equation \ref{lmi} is critical for Theorem \ref{thm:grad_dominance_param}.

\begin{assumption}\label{assump:2}
The feasible set $\mathcal{S}$ is a convex set, and 
the function $f(L,P,Z)$ is convex, bounded, and differentiable over $\mathcal{S}$. Any feasible point $(L,P,Z)\in \mathcal{S}$ is assumed to satisfy  $P\succ 0$. 
In addition, assume for all $K\in \mathcal{K}$, we can express $J(K)$ as follows\footnote{Note that this assumption needs to hold for all feasible points in the two domains, not only at the optima.},
\begin{align*}
        J(K) = &\min_{L,P,Z}\ f(L,P,Z) \\
        &\mbox{s.t.}\ 
          (L,P,Z) \in \mathcal{S},\  LP^{-1} = K.
    \end{align*}
\end{assumption} 
Recall that $J'(K,V)$ denotes the directional derivative of $J(K)$ along the direction $V$. 
When $J$ is differentiable, it holds that  $J'(K,V)=\trace(V^\tp\nabla  J(K))$.
We have the following result (modified from \cite{sun2021CDC} to also allow non-differentiable points).

\begin{theorem}\label{thm:grad_dominance_param}
Consider the problem pair in Equation \ref{f_K} and Equation \ref{lmi}. Suppose Assumption \ref{assump:2} holds, and $J(K)$ is either differentiable or subdifferentially regular in $K$.
Let $K^*$ denote a global minimizer of $J(K)$ in $\mathcal{K}$. For any $K$ satisfying $J(K)>J(K^*)$, there exists non-zero $V$ in the descent cone of $\mathcal{K}$ at $K$, such that the following inequality holds
        \begin{align}
            0< J(K) - J(K^*) \leq - J'(K,V).\label{eq:conv lin}
        \end{align}
Consequently, any stationary points of $J$ will be global minimum.
\end{theorem}
The above theorem gives a unified sufficient condition ensuring that stationary points of nonconvex PO problems are global minimum. 
For special problems such as LQR, it is possible to further bound $\norm{V}_F$ and apply $\|\nabla J(K)\|_F \ge \left|\nabla J(K)[\frac{V}{\|V\|_F}]\right|$ to show the gradient dominance property with degree $p=1$ for Equation \ref{eq:grad_dominance}.
Now suppose the convex parameterization requires a set of  parameters $P_1,\ldots,P_m$; it is possible to develop
a more general version of Theorem  \ref{thm:grad_dominance_param}, following the ideas in 
\cite{sun2021CDC},
that allows a map $K=\Phi(P_1,\ldots,P_m)$ as long as $\Phi$ has nicely-behaved first-order derivatives. A more involved version of similar  constructions has recently been proposed as a general framework of ``differentiable convex lifting" (DCL) in \cite{umenberger2022globally}. 
\begin{figure}[!ht]
	\centering 
	\begin{tabular}{c}
		\includegraphics[width=0.6\textwidth]{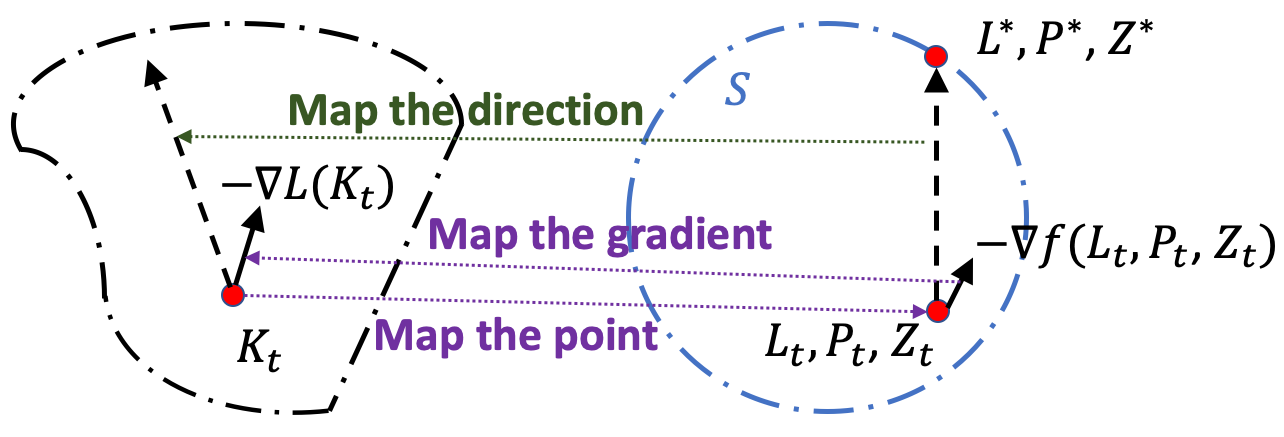}	
\end{tabular}
	\vspace{5pt}
	\caption{Proof illustration: 
 We can map any policy $K$ to $(L, P, Z)$, then map the direction $(L^*, P^*, Z^*) - (L, P, Z)$ and the gradient $\nabla f(L, P, Z)$ back to the original $K$ space. Since the problem in $(L, P, Z)$ space is convex, $\langle \nabla f(L, P, Z), (L^*, P^*, Z^*) - (L, P, Z)\rangle \leq 0$. We then show that a similar relation  holds for the nonconvex problem. 
	}
	\label{fig:1}
\end{figure} 

\noindent\textbf{Example 1:  Discrete-time infinite-horizon LQR.} 
Consider minimizing the LQR cost $J(K)$ in Equation \ref{eq:averagecost} 
subject to $K\in\mathcal{K}$, the set of all stabilizing controllers. 
This problem has the well-known convex parameterization,
\begin{subequations}\label{eq:cont_lqr_lmi}
\begin{align}
    \min_{L,P,Z}\ & f(L,P,Z):= \trace(QP) + \trace(ZR)\\
    \mbox{subject to}\ &  P\succ 0, \quad \begin{bmatrix} Z & L\\L^\top & P \end{bmatrix}\succeq 0, \quad \begin{bmatrix} P-\Sigma & AP+BL\\
    (AP+BL)^T & P
    \end{bmatrix}\succeq 0.
\end{align}
\end{subequations}
To check Assumption \ref{assump:2}, we add the constraint $K=LP^{-1}$ or $KP=L$ to problem \ref{eq:cont_lqr_lmi} and simplify the LMIs to conclude the minimum value of this problem equals $J(K)$ for all feasible $K$ \cite{sun2021CDC}. Thus Theorem \ref{thm:grad_dominance_param} applies, and can be tailored to show the gradient dominance property for the corresponding nonconvex cost $J(K)$, which allows us to recover the results seen earlier in Section \ref{sec:LQR_results} (on convergence to global minimizer) for PO in LQR. The constant in the gradient dominance inequality, Equation \ref{eq:grad_dominance}, can also be bounded in terms of problem matrices, the initial policy cost $J(K^0)$, and a notion of problem ``condition number''; see \cite{pmlr-v80-fazel18a, sun2021CDC}, and \cite{mohammadi2019global,bu2020policy} for the continuous-time counterpart. 
In fact \cite{mohammadi2019global} was the first to leverage existing LMI conditions to study the global convergence properties of PO methods, specifically for continuous time LQR. 

\vspace{5pt}\noindent\textbf{Example 2: $\mathcal{H}_\infty$ state-feedback synthesis.} As described in Section~\ref{sec:formulation},  this problem can be formulated as PO with $J(K)$ given by Equation \ref{eq:hinfcost} and $\mathcal{K}$ given by Equation \ref{equ:def_stabilizing_K}. Under the positive definite assumption on $(Q,R)$, the non-strict version of the bounded real lemma can be used to show that this problem has the following convex parameterization,
\begin{subequations}\label{eq:cont_hinf_lmi}
\begin{align*}
    \min_{L,P,\gamma}\ & f(L,P,\gamma):= \gamma\\
    \mbox{subject to}\ &  P\succ 0, \quad \bmat{ -P  &0 &P &(AP-BL)^\tp &L^\tp \\
0 &-\gamma I &0 &I &0 \\
P &0 &-\gamma Q^{-1} &0 &0 \\
AP-BL &I  &0 &-P &0\\
L &0 &0 &0 &-\gamma R^{-1}}  \preceq 0,
\end{align*}
\end{subequations}
where $\gamma$ is the closed-loop $\mathcal{H}_\infty$ norm. 
Then based on Theorem \ref{thm:grad_dominance_param} and the subdifferentially regular property of the $\mathcal{H}_\infty$ cost function, we can immediately conclude that any Clarke stationary points for this problem are global minimum \cite{Guo2022hinf}. Then, with the help of coerciveness, Goldstein's subgradient method can be guaranteed to find global minimum of this problem.

\vspace{5pt}\noindent\textbf{Example 3:  Output estimation problem with regularization.} For  lack of  coerciveness of the objective in many other control problems beyond LQR, some explicit {\it regularization} techniques can help ensure global convergence of PO methods (different from implicit regularization discussed in  \S\ref{sec:robust_control_results}).
Specifically, the recent work  \cite{umenberger2022globally} studied the continuous-time output estimation (OE), i.e., the {\it filtering} problem, a fundamental subroutine for LQG and partially observable control, and established the global convergence of PG method by regularizing the OE objective. Intriguingly, the regularization is inspired by the convex reformulation of the OE problem \citep{scherer97lmi}, and ensures that the convex {\it reformulation map}, i.e., the map  from the domain of the convex parameters to that of the policy parameter, is {\it surjective}. This way, the convex parameterization provably leads to a gradient dominance property in the policy parameter domain, of degree $1$. Hence, PO with re-balancing on the regularized OE objective is shown to enjoy a sublinear  convergence rate to the global optimum. This example not only reinforces the power of convex parameterization and its connection to gradient dominance, but also calls for more attention to ensure the non-degeneracy of the reformulation map, to actually unleash such power.

\vspace{5pt}\noindent\textbf{Discussion.} Under the framework of convex parameterization for control, 
the properties of the map between the convex and nonconvex domains, for both costs and feasible sets, are crucial. Even in the unconstrained case, if the map from the convex function to the nonconvex one introduces new stationary points, new tools will be needed to analyze whether a first-order algorithm can avoid these spurious saddle points. 
This leads to the question: what are the most general conditions on this map that preserve the stationary points of the convex function and introduce only additional ``strict saddles" or other benign stationary points? This is an interesting question for future work, for which the LQG problem is an interesting case study; see discussion in Section \ref{sec:LQG_opt_landscape}. It is also interesting to further explore the power of explicit regularization, as in \cite{umenberger2022globally} for OE problems. 
Convex parameterizations are also helpful for exploring other properties, e.g., the convex parameterization for LQG can be used to establish landscape properties of the PO with partial observation, such as the connectivity properties of feasible/sublevel sets \cite{zheng2021analysis,hu2022connect}.

\section{Challenges and Outlook} \label{challenges}

In this article, we have revisited the theoretical foundation of PO for control, and surveyed a number of results that highlight 
properties  of PO algorithms on benchmark control problems. 
Our survey 
has been inspired by recent success and wide range of applications of RL. 
PO provides a 
bridge between control and RL, and can give 
new insights into the design trade-offs between assumptions and data, as well as model-based and model-free synthesis.
Theoretical development in PO for control can help create a renewed interest in the controls community to examine synthesis of dynamic systems from this perspective, that in our view, is more integrated with machine learning. 
We close our discussion with an outlook 
on challenges and open questions 
in bridging the gap between PO theory and real-world control applications.
As an exhaustive list is impossible, we discuss a few challenges and open questions that naturally
reflect our 
perspectives. 

\vspace{5pt}\noindent\textbf{Further connections between 
optimization and control theory.} From our discussion, it is evident that the development of PO theory  requires
further connections between modern optimization theory (that focuses on convergence and complexity of iterative algorithms) with control theory (that rigorously addresses
the notions of optimality, stability, robustness, and safety for closed-loop dynamical systems). For example, it is natural to ask whether 
leveraging results in nonconvex optimization on the complexity of escaping saddle points~\cite{jin2017escape,sun2019saddlemanifold}, 
can give similar guarantees 
for PO with control-theoretic constraints.

\vspace{5pt}\noindent\textbf{Regularization for stability, robustness, and safety.}  In this article, we have covered only $\mathcal{H}_\infty$ robustness constraints.
There is an extensive system theoretic  literature on how to enforce other types of robustness and safety guarantees for control design. 
For example, more general robustness constraints can be formulated via passivity \cite{van2000l2}, dissipativity~\cite{willems72a}, or integral quadratic constraints \cite{megretski1997system}.  In addition, safety can also be induced by modifying the cost function.  
It is important to investigate how to provably pose similar robust/safety guarantees for direct policy search via either explicit regularization (on the cost/constraints) or implicit regularization (via algorithm selection).

\vspace{5pt}\noindent\textbf{Nonlinear systems, deep RL, and perception-based control.} We have focused on reviewing the PO theory centered around linear systems. It is our hope that the insights from such study can be used to guide the algorithmic/theoretical  developments of PO methods for nonlinear control.  Conceptually,  nonlinear control design can still be formulated as $\min_{K\in\mathcal{K}} J(K)$, and convergence to stationary points can still be established given coerciveness. However, how to characterize the feasible set $\mathcal{K}$ in the nonlinear control setting is   unclear in the first place. Quite often the stability/robustness constraints hold only locally for nonlinear systems. 
 It is crucial to investigate how to define and characterize feasible policies for nonlinear control problems. 
 An important class of PO problems arise in
deep RL for end-to-end perception-based control. The theoretical properties of PO methods on such problems remain largely unknown. In this case, the geometry of the feasible set can become even more complicated due to the presence of the perception modality.

\vspace{5pt}\noindent\textbf{Multi-agent systems and decentralized control.} 
Decentralized control of multi-agent systems  has a long history in control theory \cite{sandell1975survey,tsitsiklis1984problems} and also 
connects to the partially observable setting 
since each single agent cannot observe the full system state. 
PO theory for such control tasks requires further investigation.
There has been some recent progress 
\cite{feng2019exponential,li2021distributed,furieri2020learning}. For example, \cite{feng2019exponential} showed that there can be exponential number of connected component for the feasible set; \cite{furieri2020learning} then established the  global convergence and sample complexity 
under the quadratic invariance condition  \cite{rotkowitz2005characterization}. It would be interesting to explore other conditions as well as algorithm design principles that admit global convergence of PO methods for decentralized control. It is especially imperative to  develop PO methods that scale with a large number of agents. PO has also been studied under the multi-agent {\it game-theoretic} settings, 
including  the general-sum LQ dynamic games \cite{mazumdar2019policy}, with negative non-convergence results, and in LQ mean-field games  \cite{fu2019actor,carmona2019linear,wang2021global,carmona2020policy}, where the number of agents is very large and approximated by infinity. It would be interesting to further explore the PO theory in other dynamic game settings with control implications.

\vspace{5pt}\noindent\textbf{Integration of model-based and model-free methods.}  Model-based and model-free methods are both important for control design \cite{tsiamis2022statistical}.
One one hand, there is a recent trend that examines the LQR problem as a benchmark for learning-based control, starting from the work \cite{dean2017sample}, and it has been shown that model-based methods can be more sample-efficient  in  this case from an asymptotic viewpoint \citep{tu2019gap}.
On  the other hand, model-free methods can be more flexible for complex tasks such as perception-based control. As such, it is an important future direction to investigate how to integrate model-free and model-based methods to achieve the best of both worlds, especially for controlling systems which are only partially understood or parameterized. It is also expected that such an integrated approach will lead to developments in new settings that further connect learning and control theory, e.g. online  control with regret guarantees  \cite{lale2020explore,chen2021black,simchowitz2020naive,simchowitz2020improper,pmlr-v97-agarwal19c}.

\vspace{5pt}\noindent\textbf{New PO formulations from ML.}  
Many new tasks arising in machine learning for control can also be formulated as PO.
For example, imitation learning for control can be formulated as PO with control-theoretic constraints \cite{palan2020fitting,havens2021imitation,yin2021imitation,tu2022sample}.
Similarly, transfer learning for linear control can be studied as PO if we modify the cost function properly \cite{molybog2021does}. In the context of control, PO conveniently provides a general paradigm for formulating imitation learning and transfer learning tasks.
It will be interesting to investigate the convergence theory of gradient-based algorithms for such problems.

\begin{summary}[SUMMARY POINTS]
\begin{enumerate}
\item Thanks to the coerciveness and gradient dominance properties, PO for LQR leads to a nonconvex problem which can still be solved using the gradient method provably. 
\item In the state-feedback setting, advanced PO methods can be guaranteed to achieve global convergence on linear risk-sensitive/robust control tasks. 
\item In the partial observation setting, optimization landscape provides important clues for performance of PO methods.  
\item There are fundamental connections between the convex formulations of optimal/robust control tasks and the PO landscape.
\end{enumerate}
\end{summary}

\begin{issues}[FUTURE ISSUES]
\begin{enumerate}
\item Many new results on complexity of escaping saddles and finding stationary points for unconstrained optimization may be extended to PO. 
\item Advanced regularization techniques are needed for robustness and safety in general.
\item PO theory for nonlinear or perception-based control remains largely open.
\item Scalability is an important issue for PO in multi-agent decentralized control.
\item More study is needed to integrate model-based and model-free methods.
\item There are new PO problems in ML for control (e.g., imitation/transfer learning).
\end{enumerate}
\end{issues}

\section*{DISCLOSURE STATEMENT}

The authors are not aware of any affiliations, memberships, funding, or financial holdings that
might be perceived as affecting the objectivity of this review.

\vspace{-0.1in}
\section*{ACKNOWLEDGMENTS}

The research of MM is supported by grants AFOSR FA9550-20-1-0053 and NSF ECCS-2149470. MM acknowledges discussions and contributions from Jingjng Bu, Shahriar Talebi, Sham Kakade, and Rong Ge. The research of Li is supported by ONR YIP: N00014-19-1-2217, AFOSR YIP: FA9550-18-1-0150, and NSF AI Institute 2112085. Li acknowledges discussions and contributions from Yang Zheng, Yujie Tang, and Yingying Li.
The work of BH is supported by the NSF award CAREER-2048168.
BH acknowledges discussions with Peter Seiler, Geir Dullerud, Xingang Guo, Aaron Havens, Darioush Keivan, Yang Zheng, Javad Lavaei, Mihailo Jovanovic,
and Michael Overton. 
The work of KZ is supported by Simons-Berkeley Research Fellowship. KZ acknowledges discussions with  Max Simchowitz. MF acknowledges grants NSF TRIPODS II-DMS 2023166, CCF 2007036, CCF 2212261, AI Institute 2112085, as well as discussions with Yue Sun, Sham Kakade, and Rong Ge. Research of TB is supported in part by AFOSR Grant FA9550-19-1-0353, and US Army Research Laboratory (ARL) Cooperative Agreement W911NF-17-2-0196.

\vspace{-0.1in}

\bibliographystyle{ar-style3.bst}
\bibliography{main}

\end{document}